\documentclass[12pt, reqno, twoside, letterpaper]{amsart}
\usepackage[arXivon]{limit-points}

\usepackage{caption}

\setlist[enumerate]{itemsep=0pt}
\setlist[enumerate,1]{label=\textup{(}\roman*\textup{)}}

\title[Limit points and long gaps between primes]
      {Limit points and long gaps between primes}

\author[R.\ Baker]{Roger Baker}

\address{Department of Mathematics, 
         Brigham Young University, 
         Provo UT, USA.}

\email{baker@math.byu.edu}

\author[T.\ Freiberg]{Tristan Freiberg}

\address{Department of Pure Mathematics, 
         University of Waterloo, 
         Waterloo ON, CANADA.}

\email{tfreiberg@uwaterloo.ca}

\date{\today}

\begin{document}

%%%%%%%%%%%%%%%%%%%%%%%%%%%%%%%%%%%%%%%%%%%%%%%%%%%%%%%%%%%%%%%%%%
%%%%%%%%%%%%%%%%%%%%%%%%%%%% ABSTRACT %%%%%%%%%%%%%%%%%%%%%%%%%%%%
%%%%%%%%%%%%%%%%%%%%%%%%%%%%%%%%%%%%%%%%%%%%%%%%%%%%%%%%%%%%%%%%%%

\begin{abstract}
Let $d_n \defeq p_{n+1} - p_n$, where $p_n$ denotes the $n$th 
smallest prime, and let 
$R(T) \defeq \log T \log_2 T\log_4 T/(\log_3 T)^2$ (the 
``Erd{\H o}s--Rankin'' function).
We consider the sequence $(d_n/R(p_n))$ of normalized prime 
gaps, and show that its limit point set contains at least 
$25\%$ of nonnegative real numbers.
We also show that the same result holds if $R(T)$ is replaced by 
any ``reasonable'' function that tends to infinity more slowly 
than $R(T)\log_3 T$.
We also consider ``chains'' of normalized prime gaps.
Our proof combines breakthrough work of Maynard and Tao on 
bounded gaps between primes with subsequent developments of 
Ford, Green, Konyagin, Maynard and Tao on long gaps between 
consecutive primes.  
\end{abstract}

\maketitle

%%%%%%%%%%%%%%%%%%%%%%%%%%%%%%%%%%%%%%%%%%%%%%%%%%%%%%%%%%%%%%%%%%
%%%%%%%%%%%%%%%%%%%%%%%%%%%  SECTION 1 %%%%%%%%%%%%%%%%%%%%%%%%%%%
%%%%%%%%%%%%%%%%%%%%%%%%%%%%%%%%%%%%%%%%%%%%%%%%%%%%%%%%%%%%%%%%%%

\section{Introduction}
 \label{sec:intro}

Let $d_n \defeq p_{n+1} - p_n$, where $p_n$ denotes the $n$th 
smallest prime.
One variant of the prime number theorem states that
\[
 \frac{1}{x}
  \sum_{n \le x}
   \frac{d_n}{\log p_n}
    \sim 
     1
      \quad 
       (x \to \infty),
\]
that is, $d_n/\log p_n$ is approximately $1$ on average over 
$n \le x$.
As to the finer questions pertaining to the distribution of 
primes, we have little more than conjecture in the way of answers.
Heuristics based on Cram\'er's model%
\footnote{%
For details, we highly recommend the insightful expository article 
\cite{SOUND} of Soundararajan. 
}
suggest that for any given real numbers $b > a \ge 0$, 
\[
   \frac{1}{x}
  \#\big\{ n \le x : a < d_n/\log p_n \le b \big\} 
     \sim
      \int_a^b
       \e^{-t}
        \dd t
         \quad 
       (x \to \infty).
\]
However, we do not even know of any specific limit point of the 
sequence $(d_n/\log p_n)$, except for $0$ and $\infty$, 
the former having been known for just a decade, thanks to the 
groundbreaking work of Goldston--Pintz--Y{\i}ld{\i}r{\i}m 
\cite{GPY}.
(The latter follows from a 1931 result of Westzynthius 
\cite{WES}.)  

This limit point lacuna notwithstanding, Hildebrand and Maier 
\cite{HM} showed in 1988 that a positive (but unspecified) 
proportion of nonnegative real numbers are limit points of 
$(d_n/\log p_n)$.
More recently, the second author, Banks and Maynard \cite{BFM} 
have shown that in fact at least $12.5\%$ of nonnegative real 
numbers are limit points of $(d_n/\log p_n)$.
The proof strategy in \cite{BFM} incorporates an 
``Erd{\H o}s--Rankin'' type construction for producing long gaps 
between consecutive primes into the celebrated Maynard--Tao sieve, 
which was originally developed to produce short gaps between 
primes.
More recently still, Ford, Green, Konyagin, and Tao \cite{FGKT},  
and (independently) Maynard \cite{MAY2}, have settled the 
notorious ``Erd{\H o}s--Rankin problem'' by showing that $\infty$ 
is a limit point of $(d_n/R(p_n))$, where%
\footnote{%
We define $\log_2 T \defeq \log\log T$,  
$\log_3 T \defeq \log\log\log T$ and so on. 
}
$
 R(T)
  \defeq 
   \log T \log_2 T \log_4 T/(\log_3 T)^2
$.

We are therefore motivated to study limit points of 
$(d_n/R(p_n))$.
Using basically the same strategy as in \cite{BFM}, and the work 
of Ford, Green, Konyagin, and Tao \cite{FGKT}, 
Pintz \cite{PIN2,PIN3} has shown that at least $25\%$ of 
nonnegative real numbers are limit points of $(d_n/R(p_n))$.
In fact, Pintz's result is that the same statement holds if the 
normalizing function $R(T)$ is replaced by any function --- 
subject to certain technical conditions --- that tends to infinity 
no faster than $R(T)$, for example $\log T\log_2 T/(\log_3 T)^2$.

Ford, Green, Konyagin, Maynard and Tao \cite{FGKMT} have actually 
shown that, for infinitely many $n$, $d_n \gg R(p_n)\log_3 p_n$.
The purpose of this paper is to fully integrate the work of the  
five-author paper \cite{FGKMT} into the study of limit points of 
normalized prime gaps initiated in \cite{BFM, PIN2, PIN3}.
In so doing, we extend the aforementioned result of Pintz in three 
ways. 

First, we show that the normalizing function $R(T)$ may be 
replaced by any ``reasonable'' function that tends to infinity 
more slowly than $R(T)\log_3 T$, for example
$R_1(T) = \log T\log_2 T/\log_3 T$.
Second, we show that the $25\%$ may conditionally be improved to 
$33\frac{1}{3}\%$ or even $50\%$ on a certain conjecture 
concerning the level of distribution of the primes.
Third, we also consider ``chains'' of normalized, consecutive gaps 
between primes (cf.\ Theorem \ref{thm:chains}).

Precisely what we mean by a ``reasonable'' function is best 
explained in context, so we defer the statement of our main result 
to \S\ref{sec:BFM} (cf.\ Theorem \ref{thm:general}).
Examples of ``reasonable'' functions are $\log_6 T$, 
$\sqrt{\log T}$, $\log_2 T/\sqrt{\log_3 T}$, $(\log T)^{7/9}$, 
$\log T$, $R(T)$, $R_1(T)$ and $R_1(T)\log_5 T$.
Any one of these could replace $R_1(T)$ in the following special 
case of Theorem \ref{thm:general}, which will serve as a 
placeholder.

\begin{theorem}
 \label{thm:main}
Let $d_n \defeq p_{n+1} - p_n$, where $p_n$ denotes the $n$th 
smallest prime, and let $\LP[R_1]$ denote the set of limit points 
in $[0,\infty]$ of the sequence $(d_n/R_1(p_n))_{p_n \ge T_0}$, 
where 
\[
 R_1(T)
  \defeq 
   \log T \log_2 T/\log_3 T
\]
and $T_0$ is large enough so that $\log_3 T_0 \ge 1$.
Given any five nonnegative real numbers $\alpha_1,\ldots,\alpha_5$ 
with $\alpha_1 \le \cdots \le \alpha_5$, we have 
$
 \{\alpha_j - \alpha_i : 1 \le i < j \le 5\}  
  \cap 
   \LP[R_1]
    \ne 
     \emptyset.
$
\end{theorem}

As in \cite[Corollary 1.2]{BFM}, one may deduce from Theorem 1.1 
that, with $\lambda$ denoting the Lebesgue measure on $\RR$, 
\begin{equation}
\label{eq:BFM1.4}
  \lambda([0,X] \cap \LP[R_1])
   \ge 
    X/(4(1 + 1/2 + 1/3 + 1/4))
     \quad (X \ge 0), 
\end{equation} 
and%
\footnote{%
Here, by $o(1)$ we mean a positive quantity that tends to zero as 
$X$ tends to infinity.
} 
(with an ineffective $o(1)$),
\begin{equation}
\label{eq:BFM1.3}
% \[
  \lambda([0,X] \cap \LP[R_1])
   \ge 
    (1 - o(1))
     X/4
      \quad (X \to \infty).
% \]      
\end{equation}
As we will see, assuming a certain variant of the 
Elliott--Halberstam conjecture (cf.\ Hypothesis \ref{hyp:EH} 
below), one has
$
\{\alpha_2 - \alpha_1, \alpha_3 - \alpha_1, \alpha_3 - \alpha_2\} 
 \cap \LP[R_1] 
  \ne 
   \emptyset
$
for any {\em three} nonnegative real numbers 
$\alpha_1 \le \alpha_2 \le \alpha_3$, with corresponding 
improvements to \eqref{eq:BFM1.4} and \eqref{eq:BFM1.3}
(viz.\ $2 = 3 - 1$ replaces $4 = 5 - 1$).

\subsection*{Acknowledgments}

The second author gratefully acknowledges the hospitality of 
Brigham Young University, where the work on this paper commenced.

%%%%%%%%%%%%%%%%%%%%%%%%%%%%%%%%%%%%%%%%%%%%%%%%%%%%%%%%%%%%%%%%%%
%%%%%%%%%%%%%%%%%%%%%%%%%%%  SECTION 2 %%%%%%%%%%%%%%%%%%%%%%%%%%%
%%%%%%%%%%%%%%%%%%%%%%%%%%%%%%%%%%%%%%%%%%%%%%%%%%%%%%%%%%%%%%%%%%

\section{Notation and terminology}
 \label{sec:notation}

We rely heavily on the paper \cite{FGKMT} of Ford, Green, 
Konyagin, Maynard and Tao, and we follow their notation and 
conventions.
We explain these conventions here, among others, for completeness' 
sake.

\begin{enumerate}[label=---]
 \item The set of all primes is denoted by $\bP$; $p,q,s$ stand 
       for primes; $p_n$ denotes the $n$th smallest prime.
 \item For $a,b \in \ZZ$, we define 
       $a \pod{b} \defeq \{a + bc : c \in \ZZ\}$.
       Thus, $a_1 \equiv a_2 \pod{b}$ if and only if 
       $a_1 \pod{b} = a_2 \pod{b}$.         
 \item A finite set $\cH$ of integers is {\em admissible} if and 
       only if $\cH$ is not a complete set of residues modulo $p$, 
       for any prime $p$.
 \item We say an integer is {\em $x$-smooth} ($x \in \RR$) if and 
       only if its prime divisors are all less than or equal to 
       $x$.    
 \item For $n \in \ZZ$ and $\cH \subseteq \ZZ$, we define   
       $n + \cH \defeq \{n + q : q \in \cH\}$.       
 \item For statements $S$, $\ind{S} \defeq 1$ if $S$ is true and 
       $\ind{S} \defeq 0$ if $S$ is false.
 \item The cardinality of a set $\cS$ is denoted by $\#\cS$ or 
       $\#(\cS)$.
       The indicator function for $\cS \subseteq \cT$ (with $\cT$ 
       clear in context) is denoted $\ind{\cS}$.
       That is, for $t \in \cT$, 
       $\ind{\cS}(t) \defeq \ind{t \in \cS}$. 
 \item We write $\PP$ for probability and $\EE$ for expectation.
 \item Boldface symbols such as $\bX$ or $\ba$ denote random 
       variables, while non-boldface symbols such as $X$ or $a$ 
       denote their deterministic counterparts.
       Vector-valued random variables are indicated in arrowed 
       boldface, for instance 
       $\vec{\ba} = (\vec{\ba}_s)_{s \in \cS}$ denotes a random 
       tuple of random variables indexed by the set $\cS$.
 \item If $\bX$ takes at most countably many values, we define the 
       {\em essential range} of $\bX$ to be the set of all $X$ 
       such that $\PP(\bX = X) \ne 0$.
 \item If $E$ is an event of nonzero probability, 
       \[
        \PP(F \mid E)
         \defeq 
          \frac{\PP(F \land E)}{\PP(E)}
       \]
       for any event $F$, and 
       \[
        \EE(\bX \mid E)
         \defeq 
          \frac{\EE(\bX \ind{E})}{\PP(E)}
       \]
       for any absolutely integrable real-valued random variable 
       $\bX$.
       If $\bY$ is another random variable taking at most 
       countably many values, we define the conditional 
       probability $\PP(F \mid \bY)$ to be the random variable 
       that equals $\PP(F \mid \bY = Y)$ on the event $\bY = Y$ 
       for each $Y$ in the essential range of $\bY$, and similarly 
       define the conditional expectation $\EE(\bX \mid \bY)$ to 
       be the random variable that equals $\EE(\bX \mid \bY = Y)$ 
       on the event $\bY = Y$.
 \item Throughout, $x$ denotes a parameter to be thought of as 
       tending to infinity.
 \item Thus, $o(1)$ signifies a quantity that tends to zero as 
       $x \to \infty$ and $X \sim Y$ denotes that 
       $X = (1 + o(1))Y$.
 \item Expressions of the form $X = O(Y)$, $X \ll Y$ and $Y \gg X$ 
       all denote that $|X| \le c|Y|$ throughout the domain of 
       $X$, for some constant $c > 0$.
 \item The constant $c$ is to be taken as independent of any 
       parameter unless indicated otherwise, as in 
       $X \ll_{\delta,A} Y$ for instance, in which $c$ depends on 
       $\delta$ and $A$.
 \item We write $X = O_{\le}(Y)$ to denote that one can take 
       $c = 1$.
 \item We write $X \asymp Y$ to denote that $X \ll Y \ll X$.
\end{enumerate}

%%%%%%%%%%%%%%%%%%%%%%%%%%%%%%%%%%%%%%%%%%%%%%%%%%%%%%%%%%%%%%%%%%
%%%%%%%%%%%%%%%%%%%%%%%%%%%  SECTION 3 %%%%%%%%%%%%%%%%%%%%%%%%%%%
%%%%%%%%%%%%%%%%%%%%%%%%%%%%%%%%%%%%%%%%%%%%%%%%%%%%%%%%%%%%%%%%%%

\section{Proof strategy}
 \label{sec:outline}
 
Let $x$ be a large number and set 
$
 y \defeq cx\log x\log_2 x/\log_3 x
$, 
where $c > 0$ is a certain small constant.
Ford, Green, Konyagin, Maynard and Tao \cite{FGKMT} show that if 
$C$ is large enough, then there exists a vector 
$(c_p \pod{p})_{p \le Cx}$ of residue classes for which 
\[
 \big((x,y] \cap \ZZ \big) 
  \setminus \, {\textstyle \bigcup_{p \le Cx} c_p \pod{p}}
 =
   \emptyset.
\]
Thus, if $b \pod{W}$ is the residue class modulo 
$W \defeq \prod_{p \le Cx} p = \e^{(1 + o(1))Cx}$ for which 
$b \equiv -c_p \pod{p}$ for each $p \le Cx$, then for 
$n \equiv b \pod{W}$ with $n + x > Cx$, 
\[
 \bP \cap (n + x,n + y] = \emptyset.
\]

We generalize this slightly by proving that if $\cH$ is any set of 
$K$ primes in $(x,y]$ with $K \le \log x$ (say), the residue 
classes may be chosen so that 
\[
 \big((x,y] \cap \ZZ \big) 
  \setminus \, {\textstyle \bigcup_{p \le Cx} c_p \pod{p}}
  =
   \cH
\]
and hence 
\[
 \bP \cap (n + x,n + y] = \bP \cap n + \cH.
\]
Note that $\cH$, being a set of $K \le \log x$ primes greater than 
$p_K = O(K\log K)$, is admissible.  

Now let $M \ge 2$ be an integer with $M \mid K$ and 
let $\cH = \cH_1 \cup \cdots \cup \cH_M$ be a partition of $\cH$ 
into $M$ subsets of equal size.
As was shown in \cite{BFM}, with $M = 9$, a smaller choice of $y$ 
and a minor technical condition on $\cH$, the Maynard--Tao sieve 
method establishes that for large $N$, there exists 
$n \in (N,2N] \cap b \pod{W}$ and a pair $i < j$ for which 
\[
 \#(\bP \cap n + \cH_{i}), \#(\bP \cap n + \cH_{j}) \ge 1,
\] 
provided $K$ is sufficiently large and $W \le N^{\eta}$ for some 
small $\eta$.

Choosing $j - i$ to be minimal, we obtain a pair of consecutive 
primes in $n + \cH$.
We may carefully choose our primes in $\cH$ so that the spacings 
between them grow faster than $x/\log x$ but slower than $y$.

Actually, for reasons related to level of distribution and 
``Siegel'' zeros, we require that $W$ not be a multiple of a 
certain putative ``exceptional'' modulus less than
$N^{O(\eta)}$.
The largest prime divisor $p'$ of this exceptional modulus, if it 
exists, satisfies $p' \gg \log_2 N^{\eta} \gg \log x$.
For this reason, we introduce a set $\cZ$ of ``unusable'' primes, 
which has the properties of $\{p'\}$.
Their effect is negligible.

Pintz \cite{PIN3} has very recently given an elegant 
simplification of part of this argument, which we take advantage 
of in this paper, and which shows that one can take $M = 5$.

%%%%%%%%%%%%%%%%%%%%%%%%%%%%%%%%%%%%%%%%%%%%%%%%%%%%%%%%%%%%%%%%%%
%%%%%%%%%%%%%%%%%%%%%%%%%%%  SECTION 4 %%%%%%%%%%%%%%%%%%%%%%%%%%%
%%%%%%%%%%%%%%%%%%%%%%%%%%%%%%%%%%%%%%%%%%%%%%%%%%%%%%%%%%%%%%%%%%

\section{A modification of Ford--Green--Konyagin--Maynard--Tao}
 \label{sec:FGKMT}
 
\subsection{Main results}
 \label{subsec:fgkmtmain}

Given a large number $x$ we define  
\begin{equation}
 \label{eq:fgkmt3.1}
  y \defeq cx\frac{\log x \log_3 x}{\log_2 x},
\end{equation}
where $c$ is a certain (small) fixed positive constant, and 
\begin{equation}
 \label{eq:fgkmt3.2}
  z \defeq x^{\log_3 x/(4\log_2 x)}.
\end{equation}
We then define
\begin{align}
 \cS & \defeq \{\text{$s$ prime} : (\log x)^{20} < s \le z\}, \label{eq:fgkmt3.3} \\
 \cP & \defeq \{\text{$p$ prime} : x/2 < p \le x\},           \label{eq:fgkmt3.4} \\
 \cQ & \defeq \{\text{$q$ prime} : x < q \le y\}.             \label{eq:fgkmt3.5} 
\end{align}
For vectors of residue classes 
$\vec{a} \defeq (a_s \pod{s})_{s \, \in \, \cS}$ 
and 
$\vec{b} \defeq (b_p \pod{p})_{p \, \in \, \cP}$, 
we define sifted sets
\[
  S(\vec{a})
  \defeq 
   \ZZ 
    \, \setminus \, 
     \textstyle{\bigcup_{s \, \in \, \cS}} \,
      a_s \pod{s}
\quad 
 \text{and}
  \quad 
  S(\vec{b})
  \defeq 
   \ZZ 
    \, \setminus \, 
     \textstyle{\bigcup_{p \, \in \, \cP}} \,
      b_p \pod{p}.
\]      
We note that in view of the prime number theorem (with suitably 
strong error term) and \eqref{eq:fgkmt3.1},   
\begin{equation}
 \label{eq:Qsize}
 \#\cQ 
  = 
   \frac{y}{\log x}
    \bigg(
     1 + O\bigg(\frac{\log_2 x}{\log x}\bigg)
    \bigg).
\end{equation}
Finally, let $K$ be any natural number satisfying 
\begin{equation}
 \label{eq:Kbnd}
  K \le \log x.
\end{equation}
By \eqref{eq:Qsize}, we may suppose $x$ is large enough so that 
$\cQ$ contains at least $K$ primes.
Since $p_K \ll K\log K$, we may also suppose that $p_K \le x$.
We fix any 
\begin{equation}
 \label{eq:cHdef}
 \cH 
  \defeq \{q_1,\ldots,q_K\}
   \subseteq \cQ 
  \quad 
   \text{with}
    \quad 
   \#\cH = K.
\end{equation}
Note that $\cH$, being a set of $K$ primes larger than $p_K$, is 
an admissible set.

\begin{theorem}[Sieving for primes]
 \label{thm:fgkmt2}
For all sufficiently large $x$, there exist vectors of residue 
classes 
$\vec{a} = (a_s \pod{s})_{s \, \in \, \cS}$ 
and 
$\vec{b} = (b_p \pod{p})_{p \, \in \, \cP}$ 
such that $\cH \subseteq S(\vec{a}) \cap S(\vec{b})$ and 
\begin{equation}
 \label{eq:fgkmt3.6}
  \#(\cQ \cap S(\vec{a}) \cap S(\vec{b}))
   \ll
    \frac{x}{\log x}, 
\end{equation}
where the implied constant is absolute.
\end{theorem}

The only difference between Theorem \ref{thm:fgkmt2} and 
\cite[Theorem 2]{FGKMT} is our additional requirement that 
$\cH \subseteq S(\vec{a}) \cap S(\vec{b})$.
Unsurprisingly, the proof of Theorem \ref{thm:fgkmt2} follows that 
of \cite[Theorem 2]{FGKMT} very closely, even verbatim in many 
parts.
Nevertheless, the details must be checked, and by including them 
here we are also able to point out the minor differences between 
the two proofs.

We now introduce a set $\cZ$ of ``unusable'' primes with the 
property that for any $p' \in \cZ$,
\begin{equation}
 \label{eq:Zsparse}
  \sums[p \ge p'][p \in \cZ]
   \frac{1}{p}
    \ll
     \frac{1}{p'}
      \ll
       \frac{1}{\log x}.
\end{equation}

\begin{corollary}
 \label{cor:thm2}
Let $C$ be a sufficiently large but fixed positive constant.
For all sufficiently large $x$, there exists a vector of residue 
classes 
$(c_p \pod{p})_{p \le Cx, \, p \, \not\in \, \cZ}$ such that 
$
 \cH
  =
   \big(\ZZ \cap (x,y]\big)
    \setminus \, 
     {\textstyle \bigcup_{p \le Cx, \, p \, \not\in \, \cZ}} 
       \, c_p \pod{p}.
$ 
\end{corollary}

We deduce Corollary \ref{cor:thm2} from Theorem \ref{thm:fgkmt2} 
with the aid of Lemma 5.1 of \cite{BFM}, which is as follows.

\begin{lemma}
 \label{lem:BFM5.1}
Let $\sH,\sT$ be sets of integers, $\sP$ a set of 
primes, such that for some $x \ge 2$, 
$\sH \subseteq \sT \subseteq [0,x^2]$ and 
$
 \#\{p \in \sP : p > x\} 
  > 
   \#\sH + \#\sT
$.
If $\sH$ is admissible then there exists a vector of residue 
classes $(\gamma_p \pod{p})_{p \, \in \, \sP}$ such that 
$
 \sH
  = 
   \sT 
    \, \setminus \,
     {\textstyle \bigcup_{p \, \in \, \sP}}\, \gamma_p \pod{p}.
$
\end{lemma}

\begin{proof}[Deduction of Corollary \ref{cor:thm2}]
We choose $x$, $\vec{a}$ and $\vec{b}$ so that the conclusions 
of Theorem \ref{thm:fgkmt2} hold, and work with the enlarged 
sifted sets
\[
  S_{\cZ}(\vec{a})
  \defeq 
   \ZZ 
    \, \setminus \, 
     \textstyle{\bigcup_{s \, \in \, \cS \, \setminus \, \cZ}} \,
      a_s \pod{s}
\quad 
 \text{and}
  \quad 
  S_{\cZ}(\vec{b})
  \defeq 
   \ZZ 
    \, \setminus \, 
     \textstyle{\bigcup_{p \, \in \, \cP \, \setminus \, \cZ}} \,
      b_p \pod{p}.
\]   
Note that if $n \in S_{\cZ}(\vec{a}) \cap S_{\cZ}(\vec{b})$, then 
either $n \in S(\vec{a}) \cap S(\vec{b})$, 
$n \equiv a_s \pod{s}$ for some $s \in \cS \cap \cZ$ 
or 
$n \equiv b_p \pod{p}$ for some $p \in \cP \cap \cZ$.
Now,  
\[
 \sum_{s \, \in \, \cS \cap \cZ}
  \sums[n \le y][n \equiv a_{s} \pod{s}]
   1
 \, +
    \sum_{p \, \in \, \cP \cap \cZ}
     \sums[n \le y][n \equiv b_{p} \pod{p}]
      1
    \le 
     y
      \sums[p \, \in \, \cZ][p > (\log x)^{20}]
       \frac{1}{p'}
        \ll
          \frac{y}{(\log x)^{20}}
\]
by \eqref{eq:fgkmt3.3}, \eqref{eq:fgkmt3.4} and 
\eqref{eq:Zsparse}.
Thus, the elements of 
$\cQ \cap S_{\cZ}(\vec{a}) \cap S_{\cZ}(\vec{b})$
that are not in 
$\cQ \cap S(\vec{a}) \cap S(\vec{b})$ 
number at most $y/(\log x)^{20} \ll x/(\log x)^{19}$ 
by \eqref{eq:fgkmt3.1}.
We conclude from \eqref{eq:fgkmt3.6} that 
\begin{equation}
 \label{eq:QZbnd}
  \#(\cQ \cap S_{\cZ}(\vec{a}) \cap S_{\cZ}(\vec{b}))
   \ll
    \frac{x}{\log x}.
\end{equation}

Let 
\[
 \cL 
  \defeq 
   \{
     \text{$\ell$ prime} : 
      \ell \in [2,(\log x)^{20}] \cup (z,x/2] 
   \}
\]
so that $\cS \cup \cL \cup \cP$ is a partition of the primes less 
than or equal to $x$.
We define a vector of residue classes 
$(c_p \pod{p})_{p \le x, \, p \not\in \cZ}$ 
by setting 
\begin{align*}
 c_p  
  \defeq 
   \begin{cases}
    a_p  & p \in \cS \, \setminus \, \cZ \\
    b_p  & p \in \cP \, \setminus \, \cZ \\
    0    & p \in \cL \, \setminus \, \cZ. 
   \end{cases}
\end{align*}
Recalling that Theorem \ref{thm:fgkmt2} gives 
$\cH \subseteq S(\vec{a}) \cap S(\vec{b})$, and noting that $\cH$ 
consists of primes larger than $x$ by definition, we see that
\[
 \cH
  \subseteq
  \cT
  \defeq 
   \big(\ZZ \cap (x,y]\big)
    \setminus \, 
     {\textstyle \bigcup_{p \le x, \, p \not\in \cZ}} \,
      c_p \pod{p}.
\]

Now, if $n \in \cT$ then $x < n \le y$ and either 
\begin{enumerate}[label=(\arabic*)]
 \item $n$ is divisible by a prime $p' > x/2$, 
 \item $n$ is divisible by a prime $p' \in (z,x/2]$, or
 \item $n$ is $z$-smooth.
\end{enumerate}

In case (1), $n = mp'$ for some 
$m \le y/p' < 2y/x = o(\log x)$ by \eqref{eq:fgkmt3.1} and so, by 
\eqref{eq:Zsparse}, $m$ is not divisible by any prime in $\cZ$ 
(provided $x$ is sufficiently large, as we assume).
Nor is $m$ divisible by any other prime $p < 2y/x$, since 
such primes are in $\cL \, \setminus \, \cZ$ and $c_{p} = 0$ for 
such primes.
Hence $m = 1$, and $n$ must belong to 
$\cQ \cap S_{\cZ}(\vec{a}) \cap S_{\cZ}(\vec{b})$.
Thus,
\begin{equation}
 \label{eq:(1)bnd}
  \#\{n \in \cT : \text{(1) holds} \} 
   \le 
    \#(\cQ \cap S_{\cZ}(\vec{a}) \cap S_{\cZ}(\vec{b}))
     \ll
      \frac{x}{\log x}
\end{equation}
by \eqref{eq:QZbnd}.
In case (2), the prime $p'$ must belong to $\cZ$, for otherwise 
$c_{p'} = 0$. 
Thus, 
\begin{equation}
 \label{eq:(2)bnd}
  \#\{n \in \cT : \text{(2) holds} \}
 \ll
 y
  \sums[p > z, \, p \in \cZ] 
   \frac{1}{p}
    \ll
     \frac{y}{z}
    = 
     o\Big(\frac{x}{\log x}\Big)
\end{equation}
by \eqref{eq:fgkmt3.2} and \eqref{eq:fgkmt3.1}.
As shown in \cite[Theorem 2 et seq.]{FGKMT}, smooth number 
estimates give 
\begin{equation} 
 \label{eq:(3)bnd}
 \#\{n \in \cT : \text{(3) holds} \}
 = 
  o\Big(\frac{x}{\log x}\Big).
\end{equation}

Combining \eqref{eq:(1)bnd}, \eqref{eq:(2)bnd} and 
\eqref{eq:(3)bnd}, we obtain $K + \#\cT \ll x/\log x$ in view of 
\eqref{eq:Kbnd}.
We may therefore choose our constant $C$ to be large enough so 
that
\[
 \#\{\text{$p$ prime}: p \in (x,Cx], \, p \not\in \cZ \} 
  >
   K + \#\cT.
\] 
(The number of primes in $\cZ$ that belong to $(x,Cx]$ is 
negligible, for 
\[
  \#\{p \in \cZ : p \le Cx\} \ll \log x,
\]
as can be seen from \eqref{eq:Zsparse} [write $1 = p/p$ and sum 
dyadically].)
As $\cH$ is an admissible subset of $\cT \subseteq (x,y]$, we 
must conclude, in view of Lemma \ref{lem:BFM5.1}, that for 
sufficiently large $x$ there exist residue classes $c_p \pod{p}$ 
for $p \in (x,Cx]$, $p \not\in \cZ$, such that 
\[
 \cH
  =
   \cT
    \, \setminus \,
    {\textstyle \bigcup_{p \, \in \, (x,Cx], \, p \, \not\in \, \cZ}} 
      \, c_p \pod{p}
  =
   \big(\ZZ \cap (x,y]\big)
    \setminus \, 
     {\textstyle \bigcup_{p \le Cx, \, p \, \not\in \, \cZ}} 
       \, c_p \pod{p}.
\] 
\end{proof}

In order to prove Theorem \ref{thm:fgkmt2} we must first establish 
the following result, which is analogous to 
\cite[Theorem 4]{FGKMT}. 

\begin{theorem}[Random construction]
 \label{thm:fgkmt4}
Let $x$ be sufficiently large.
There exists a positive number $C$ with 
\begin{equation}
 \label{eq:fgkmt4.29}
  C \asymp \frac{1}{c},
\end{equation}
the implied constants being independent of $c$, a set of 
positive integers $\{h_1,\ldots,h_r\}$ with $r \le \sqrt{\log x}$, 
and random vectors  
$\vec{\ba} = (\ba_s \pod{s})_{s \, \in \, \cS}$ 
and 
$\vec{\bn} = (\bn_p)_{p \, \in \, \cP}$
of residue classes $\ba_s \pod{s}$ and integers $\bn_p$ 
respectively, satisfying the following.
\begin{enumerate}
 \item For every $\vec{a} = (a_s \pod{s})_{s \in \cS}$ in the 
       essential range of $\vec{\ba}$, we have 
       \[
        \cH \cap a_s \pod{s} = \emptyset \quad  (s \in \cS).
       \]
 \item For every $\vec{n} = (n_p \pod{p})_{p \in \cP}$ in the 
       essential range of $\vec{\bn}$, we have 
       \[
        \cH \cap n_p \pod{p} = \emptyset \quad (p \in \cP).
       \]
 \item For every $\vec{a}$ in the essential range of $\vec{\ba}$, 
       we have
       \[
        \PP(q \in \be_p(\vec{a}) \mid \vec{\ba} = \vec{a})
         \le 
          x^{-3/5}
           \quad 
            (p \in \cP),
       \]
       where 
       $
        \be_p(\vec{a}) \defeq \{\bn_p + h_ip : i \le r\}
         \cap 
          \cQ 
           \cap 
            S(\vec{a})
       $.
  \item With probability $1 - o(1)$, we have  
        \begin{equation}
         \label{eq:fgkmt4.30}
          \#(\cQ \cap S(\vec{\ba}))
           \sim 
            80cx\frac{\log_2 x}{\log x}.
        \end{equation}
  \item Call an element $\vec{a}$ in the essential range of 
        $\vec{\ba}$ ``good'' if, for all but at most 
        $\frac{x}{\log x \log_2 x}$ elements 
        $q \in \cQ \cap S(\vec{a})$, one has 
        \begin{equation}
         \label{eq:fgkmt4.31}
          \sum_{p \in \cP}
           \PP(q \in \be_p(\vec{a}) \mid \vec{\ba} = \vec{a})
            =
             C + O_{\le}\bigg(\frac{1}{(\log_2 x)^{2}}\bigg).
        \end{equation}
        Then $\vec{\ba}$ is good with probability $1 - o(1)$.
\end{enumerate}
\end{theorem}

\subsection{Proof of Theorem \ref{thm:fgkmt4}}
 \label{subsec:thm4pf}

Let $x,c,y,z,\cS,\cP,\cQ,K$ and $\cH$ be as in 
Theorem \ref{thm:fgkmt4}.
We set
\begin{equation}
 \label{eq:fgkmt6.8}
  r \defeq \lfloor (\log x)^{1/5} \rfloor
\end{equation}
and let $\{h_1,\ldots,h_r\}$ be the admissible set with 
$h_i \defeq (2i-1)^2$ for $i \le r$.
Our first lemma is a special case of \cite[Theorem 5]{FGKMT}.
\begin{lemma}[Existence of a good sieve weight]
 \label{thm:fgkmt5}
There exist positive quantities
\begin{equation}
 \label{eq:fgkmt6.2}
  \tau \ge x^{o(1)} 
   \quad 
    \text{and}
     \quad 
    u \asymp \log_2 x,
\end{equation}
and a function $w : \cP \times \ZZ \to \RR^+$ supported on 
$\cP \times [-y,y]$, satisfying the following.
\begin{enumerate} 
 \item Uniformly for $p \in \cP$,  
\begin{equation}
 \label{eq:fgkmt6.4}
  \sum_{n \in \ZZ} w(p,n)
   =
    \bigg(
       1 + O\bigg(\frac{1}{(\log_2 x)^{10}}\bigg)
    \bigg)
     \tau
      \frac{y}{(\log x)^r}.
\end{equation}
 \item Uniformly for $q \in \cQ$ and $i \le r$,   
\begin{equation}
 \label{eq:fgkmt6.5}
  \sum_{p \in \cP}
   w(p,q - h_ip)
    =
    \bigg(
      1 + O\bigg(\frac{1}{(\log_2 x)^{10}}\bigg)
    \bigg)
     \tau
      \frac{u}{r} 
       \frac{x}{2(\log x)^r}.
\end{equation}  
 \item Uniformly for $(p,n) \in \cP\times \ZZ$,   
\begin{equation}
 \label{eq:fgkmt6.7}
  w(p,n) = O\big(x^{1/3 + o(1)}\big).
\end{equation}
\end{enumerate}
\end{lemma}

We choose $\tau$, $u$ and $w : \cP \times \ZZ \to \RR^+$ 
according to Lemma \ref{thm:fgkmt5}, and define 
$w_{\cH} : \cP \times \ZZ \to \RR^+$ by setting 
\[
 w_{\cH}(p,n) 
  \defeq 
   \begin{cases}
    w(p,n) & \text{if $\cH \cap n \pod{p} = \emptyset$} \\
    0      & \text{otherwise.}  
   \end{cases}
\]

\begin{lemma}
 \label{lem:rcb2}
Statements \textup{(}i\textup{)}, \textup{(}ii\textup{)} and 
\textup{(}iii\textup{)} of Lemma \ref{thm:fgkmt5} all hold with 
$w_{\cH}$ in place of $w$, provided the hypothesis $q \in \cQ$ in 
\textup{(}ii\textup{)} is replaced by the hypothesis 
$q \in \cQ\setminus \cH$.
\end{lemma}

\begin{proof}
We only need to consider (i) and (ii).
For every $p \in \cP$ we have 
\begin{align*}
 0 \le \sum_{n \in \ZZ} (w(p,n) - w_{\cH}(p,n))
    \le \sum_{j=1}^K \sums[|n| \le y][n \equiv q_j \pod{p}] w(p,n)
     \ll x^{1/3 + o(1)}y/p
      \ll x^{-2/3 + o(1)}y,
\end{align*}
which gives the analog of \eqref{eq:fgkmt6.4} for $w_{\cH}$ in 
view of \eqref{eq:fgkmt6.2}.
For every $q \in \cQ\setminus \cH$ and $i \le r$, we see 
similarly that 
\begin{align*}
 0 
  & \le 
     \sum_{p \in \cP} (w(p,q - h_ip) - w_{\cH}(p,q - h_ip)) 
  \\
  & \hspace{45pt}
    \le
     x^{1/3 + o(1)}
      \sum_{j=1}^K
       \sum_{q - h_ip \equiv q_j \pod{p}} 1 
    \le 
     x^{1/3 + o(1)}
      \sum_{j=1}^K
       \sum_{p \mid q - q_j} 1 
    \ll
     x^{1/3 + o(1)},
\end{align*}
which gives the analog of \eqref{eq:fgkmt6.5}.
\end{proof}

For each $p \in \cP$, let $\tilde{\bn}_p$ denote the random 
integer with probability density 
\[
  \PP(\tilde{\bn}_p = n) 
   \defeq 
    \frac{w_{\cH}(p,n)}{\sum_{n' \in \ZZ} w_{\cH}(p,n')}
\]
for all $n \in \ZZ$.
Using Lemma \ref{lem:rcb2}, we verify that  
\begin{equation}
 \label{eq:fgkmt6.9}
  \sum_{p \in \cP}
   \PP(q = \tilde{\bn}_p + h_ip)
   =
    \ind{\cQ\setminus\cH}(q)
%     \ind{(q \in \cQ\setminus \cH)}
    \bigg(
     1 + O\bigg(\frac{1}{(\log_2 x)^{10}}\bigg)
    \bigg)
      \frac{u}{r}
       \frac{x}{2y}
  \quad 
   (q \in \cQ, i \le r),
\end{equation}
\begin{equation}
 \label{eq:fgkmt6.10}
  \PP(\tilde{\bn}_p = n) \ll x^{-2/3 + o(1)}
   \quad (p \in \cP, n \in \ZZ)
\end{equation}
and
\begin{equation}
 \label{eq:probnph0}
  \PP(\cH \cap \tilde{\bn}_p \pod{p} \ne \emptyset) = 0
   \quad (p \in \cP).
\end{equation}
By \eqref{eq:fgkmt6.10}, the analog of \eqref{eq:fgkmt6.9} holds 
with a single prime deleted from $\cP$.

We choose the random vector 
$\vec{\ba} \defeq (\ba_s \pod{s})_{s \in \cS}$ by selecting each 
$\ba_s \pod{s}$ uniformly at random from 
\begin{equation}
 \label{eq:defOmega}
  \Omega_{\cH}(s)
   \defeq 
    (\ZZ/s\ZZ)\setminus\{q \pod{s} : q \in \cH\},
\end{equation}
independently in $s$ and independently of the $\tilde{\bn}_p$.
Note, then, that for any random vector $\vec{\ba}$, 
$\cH \subseteq S(\vec{\ba})$.

The sifted set $S(\vec{\ba})$ is a random periodic subset of $\ZZ$ 
with density 
\[
 \sigma_{\cH}
  \defeq 
   \prod_{s \in \cS}
    \bigg(
     1 - \frac{1}{\#\Omega_{\cH}(s)}
    \bigg).
\]
Let us compare $\sigma_{\cH}$ with the quantity 
\[
 \sigma 
  \defeq 
   \prod_{s \in \cS}
    \bigg(
     1 - \frac{1}{s}
    \bigg)
\]
defined in \cite{FGKMT}. 
As $(\log x)^{20} - \log x \le s - K \le \#\Omega_{\cH}(s) \le s$ 
(cf.\ \eqref{eq:fgkmt3.3} and \eqref{eq:Kbnd}), one may verify, in 
a straightforward manner, that 
\begin{equation}
 \label{eq:sigmahatsigma}
  \sigma_{\cH} 
   = 
    \sigma\bigg(1 + O\bigg(\frac{1}{(\log x)^{19}}\bigg)\bigg). 
\end{equation}
Consequently, in the estimates that follow, $\sigma_{\cH}$ and 
$\sigma$ are interchangeable. 

As noted in \cite{FGKMT}, by the prime number theorem (with 
suitably strong error term), \eqref{eq:fgkmt3.2} and 
\eqref{eq:fgkmt3.3}, 
\[
 \sigma 
  = 
   \bigg(
    1 + O\bigg(\frac{1}{(\log_2 x)^{10}}\bigg)
   \bigg)
    \frac{80\log_2 x}{\log x \log_3 x/\log_2 x},
\]
so by \eqref{eq:fgkmt3.1} we have 
\begin{equation}
 \label{eq:fgkmt6.11}
  \sigma y
   =
   \bigg(
    1 + O\bigg(\frac{1}{(\log_2 x)^{10}}\bigg)
   \bigg)    
    80c x\log_2 x.
\end{equation}
Also, by \eqref{eq:fgkmt6.8} we have 
\begin{equation}
 \label{eq:fgkmt6.12}
  \sigma^r = x^{o(1)}.
\end{equation}

Let
\begin{equation}
 \label{eq:fgkmt6.14}
 X_p(\vec{a})
  \defeq 
   \PP(\tilde{\bn}_p + h_ip \in S(\vec{a}) \, \, \hbox{for all} \, \, i \le r)
\end{equation}
and
\begin{equation}
 \label{eq:rcb11a}
  \cP(\vec{a})
     \defeq 
      \bigg\{
       p \in \cP : 
        X_p(\vec{a})
   =
    \bigg(
     1 + O_{\le}\bigg(\frac{1}{(\log x)^{6}}\bigg)
    \bigg)
     \sigma^r
      \bigg\}.
\end{equation}
It will transpire that with probability $1 - o(1)$, most primes in 
$\cP$ lie in $\cP(\vec{\ba})$.

We now define $\bn_p$ in a slightly complicated way.
Let 
\[
 Z_p(\vec{a};n)
  \defeq 
   \ind{
    (n + h_j p \, \in \, S(\vec{a}) \, \, \forall j \le r)
       }
        \PP(\tilde{\bn}_p = n).
\]
Suppose we are in the event that $\vec{\ba} = \vec{a}$.
If $p \in \cP \, \setminus \, \cP(\vec{a})$, we set 
$\bn_p = 0$.
Otherwise, let $\bn_p$ be the random integer with 
\begin{equation}
 \label{eq:rcb11}
  \PP(\bn_p = n \mid \vec{\ba} = \vec{a})
   =
    \frac{Z_p(\vec{a};n)}{X_p(\vec{a})},
\end{equation}
with the $\bn_p$ jointly conditionally independent on the 
event $\vec{\ba} = \vec{a}$. 
(We easily verify that 
$
 \sum_{n \in \ZZ} Z_p(\vec{a};n) = X_p(\vec{a})
$,
so that \eqref{eq:rcb11} makes sense.)

\begin{proof}%
[Deduction of Theorem \ref{thm:fgkmt4} 
 \textup{(}i\textup{)} -- \textup{(}iii\textup{)}]
Let $p \in \cP$.
We claim that 
\[
 \PP(\cH \cap \bn_p \pod{p} \ne \emptyset) = 0.
\]
To prove the claim it suffices to show 
$
 \PP(\cH \cap \bn_p \pod{p} \ne \emptyset \mid \vec{\ba} = \vec{a}) 
  = 0
$
for every $\vec{a}$.
This is easily checked if $p \not\in \cP(\vec{a})$, since 
$\cH \cap 0 \pod{p} \subseteq \cQ \cap 0 \pod{p} = \emptyset$; 
otherwise 
\[
 \PP(\cH \cap \bn_p \pod{p} \ne \emptyset \mid \vec{\ba} = \vec{a})
  \le 
   \frac{\PP(\cH \cap \tilde{\bn}_p \pod{p} \ne \emptyset)}
        {X_p(\vec{a})}
    =
     0.
\]

We see that Theorem \ref{thm:fgkmt4} (i) and (ii) hold and we can 
now prove Theorem \ref{thm:fgkmt4} (iii).
Given $\vec{a}$ in the essential range of $\vec{\ba}$, and 
$q \in \cQ \cap S(\vec{a})$, we have 
\[
 \PP(q \in \be_p(\vec{a}) \mid \vec{\ba} = \vec{a})
  \le 
   \PP( \bn_p + h_ip = q  \, \, \hbox{for some} \, \, i \le r\mid \vec{\ba} = \vec{a}).
\]
The right-hand side is $0$ if $p \not\in \cP(\vec{a})$.
Otherwise,
\begin{align*}
 \PP(q \in \be_p(\vec{a}) \mid \vec{\ba} = \vec{a})
  & \le 
   r \max_{n \in \ZZ} \PP(\bn_p = n \mid \vec{\ba} = \vec{a})  
    \ll
     r\sigma^{-r} \max_{n \in \ZZ} \PP(\tilde{\bn}_p = n) \\
   &   \ll
       x^{-2/3 + o(1)}.
\end{align*}
(Here, we have used \eqref{eq:fgkmt6.8}, \eqref{eq:fgkmt6.10} and 
\eqref{eq:fgkmt6.12}.)
\end{proof}

The following lemma is analogous to \cite[Lemma 6.1]{FGKMT}.

\begin{lemma}
 \label{lem:rcb4}
Let $n_1,\ldots,n_t$ be distinct integers of magnitude $x^{O(1)}$, 
$t \le \log x$, such that 
$\cH \cap \{n_1,\ldots,n_t\} = \emptyset$.
Then for all sufficiently large $x$, 
\[
 \PP(n_1,\ldots,n_t \in S(\vec{\ba}))
  =
   \bigg(1 + O\bigg(\frac{1}{(\log x)^{16}}\bigg)\bigg)
    \sigma^t,
\]
where the implied constant is absolute.
\end{lemma}
\begin{proof}
% % % For $s \in \cS$, we're choosing $\ba_s \pod{s}$ uniformly at 
% % % random from $\Omega_{\cH}(s)$, independently 
% % % in $s$.
% % % %
% % % Therefore, the probability that 
% % % $\ba_s \pod{s} \not \in \{n_1 \pod{s},\ldots,n_t \pod{s}\}$ is
% % % \[
% % %  1 - \frac{t_{s}}{\#\Omega_{\cH}(s)},
% % %   \quad 
% % %  t_{s} 
% % %      \defeq 
% % %    |\{n_1 \pod{s},\ldots,n_t \pod{s}\} \cap \Omega_{\cH}(s)|,   
% % % \]
% % % and
For $s \in \cS$, let 
$
 t_{s}
  \defeq 
  \#(\Omega_{\cH}(s) \cap \{n_1 \pod{s},\ldots,n_t \pod{s}\}) 
$.
We have 
\[
 \PP(n_1,\ldots,n_t \in S(\vec{\ba}))
  =
  \prod_{s \in \cS}
   \bigg(1 - \frac{t_{s}}{\#\Omega_{\cH}(s)}\bigg).
\]
%
% % % Note that, since $0 \le t_{s} \le t \le \log x$ and 
% % % $s - |\cH\pod{s}| \ge s - K \gg (\log x)^{20}$ for 
% % % $s \in \cS$, we have
% % % \[
% % %  1 - \frac{t_{s}}{\#\Omega_{\cH}(s)}
% % %   = 
% % %    1 + O\br{t/s}
% % %     =
% % %     1 + O\br{(\log x)^{-19}}.
% % % \]
%
Note that for $s \in \cS$, 
$
 1 - t_{s}/\#\Omega_{\cH}(s)
  = 
   1 + O\br{t/s}
    =
    1 + O\br{1/(\log x)^{19}}
$.

Let $\cS'$ be the set of primes $s \in \cS$ such that either 
$n_i \equiv n_j \pod{s}$ for some $i \ne j$ or 
$\cH \cap n_i \pod{s} \ne \emptyset$ for some $i$.
For $i \ne j$ we have $1 \le |n_i - n_j| \ll x^{O(1)}$, so 
$n_i - n_j$ has at most $O(\log x)$ prime divisors.
Similarly, for each $i$ and $j$, $n_i - q_j$ has at most 
$O(\log x)$ prime divisors.
We see that $|\cS'| \ll (t^2 + tK)\log x \ll (\log x)^3$ 
(cf.\ \eqref{eq:Kbnd}), and 
\begin{align*}
 & 
 \prod_{s \in \cS'}
  \bigg(1 - \frac{t}{\#\Omega_{\cH}(s)}\bigg)^{-1}
   \bigg(1 - \frac{t_{s}}{\#\Omega_{\cH}(s)}\bigg)
 \\
 & \hspace{30pt} 
  =
   \prod_{s \in \cS'}
    \bigg(1 + O\bigg(\frac{1}{(\log x)^{19}}\bigg)\bigg)^{|\cS'|}
  =
   1 + O\bigg(\frac{1}{(\log x)^{16}}\bigg).
\end{align*}
For $s \in \cS \, \setminus \, \cS'$ we have $t_{s} = t$.
Thus, 
\begin{align*}
 \prod_{s \in \cS}
  \bigg(1 - \frac{t_{s}}{\#\Omega_{\cH}(s)}\bigg)
 & = 
    \bigg(1 + O\bigg(\frac{1}{(\log x)^{16}}\bigg)\bigg)
     \prod_{s \in \cS}
      \bigg(1 - \frac{t}{\#\Omega_{\cH}(s)}\bigg)
 \\
 & =
     \bigg(1 + O\bigg(\frac{1}{(\log x)^{16}}\bigg)\bigg) 
      \sigma^t
      \prod_{s \in \cS}
       \bigg(1 + O\bigg(\frac{t^2}{s^2}\bigg)\bigg) 
 \\
 & = 
  \sigma^t
   \bigg(1 + O\bigg(\frac{1}{(\log x)^{16}}\bigg)\bigg).
\end{align*}
\end{proof}

\begin{proof}%
[Deduction of Theorem \ref{thm:fgkmt4} \textup{(}iv\textup{)}]
Let $\cR \defeq \cQ\setminus \cH$.
Recalling \eqref{eq:cHdef} and \eqref{eq:Kbnd}, we have   
\begin{equation}
 \label{eq:rcb13}
 \#(\cQ \cap S(\vec{\ba})) 
  = 
  \#(\cR \cap S(\vec{\ba})) + K 
   =
    \#(\cR \cap S(\vec{\ba})) + O_{\le}(\log x).
\end{equation}
Let 
$
 X 
  \defeq 
   \sum_{q \in \cR} \ind{q \in S(\vec{\ba})}.  
$
We have
\[
  \EE X 
   =
    \sum_{q \in \cR}
     \PP(q \in S(\vec{\ba}))
   =
    (\#\cR)\sigma
      \bigg(1 + O\bigg(\frac{1}{(\log x)^{16}}\bigg)\bigg)
\]
from Lemma \ref{lem:rcb4}.
Note that by \eqref{eq:Qsize} we have 
\[
 (\#\cQ)\sigma
  = 
    \frac{\sigma y}{\log x}
     \bigg(1 + O\bigg(\frac{\log_2 x}{\log x}\bigg)\bigg).
\]
By \eqref{eq:rcb13}, the same estimate holds for $(\#\cR)\sigma$,
and
\begin{equation}
 \label{eq:rcb14}
  \EE \#(\cQ \cap S(\vec{\ba}))
   =
    \frac{\sigma y}{\log x}
     \bigg(1 + O\bigg(\frac{\log_2 x}{\log x}\bigg)\bigg).
\end{equation}
We similarly have 
\begin{align*}
 \EE X^2
  & =
   \sum_{q_1 \in \cR}
    \sum_{q_2 \in \cR}
     \ind{q_1,q_2 \in S(\vec{\ba})}
 \\
  & = 
   \sums[(q_1,q_2) \in \cR^2][q_1 \ne q_2]
    \sigma^2
     \bigg(1 + O\bigg(\frac{1}{(\log x)^{16}}\bigg)\bigg)
   + 
     \sum_{q \in \cR}
      \sigma
       \bigg(1 + O\bigg(\frac{1}{(\log x)^{16}}\bigg)\bigg)
  \\
  & = 
   \sigma^2(\#\cR)^2
    \bigg(1 + O\bigg(\frac{1}{(\log x)^{16}}\bigg)\bigg).
\end{align*}
Thus,  
\begin{equation}
 \label{eq:rcb15}
  \EE (X - \EE X)^2 
   =
    \EE X^2 - (\EE X)^2
   \ll
    \frac{\sigma^2(\#\cR)^2}{(\log x)^{16}}.
\end{equation}

Now we use Chebyshev's inequality: 
\begin{align}
 \begin{split}
  \label{eq:rcb16}
  \PP\big(|X - \EE X| > (\#\cR)\sigma(\log x)^{-3}\big)
   & 
    \le 
    (\#\cR)^{-2}\sigma^{-2}(\log x)^6 \, \EE (X - \EE X)^2
  \\
  & 
   \ll
    \frac{1}{(\log x)^{10}}.
 \end{split}
\end{align}
Recalling \eqref{eq:rcb13} again we get 
\[
 \#(\cQ \cap S(\vec{\ba}))
  =
   \frac{\sigma y}{\log x}
    \bigg(1 + O\bigg(\frac{\log_2 x}{\log x}\bigg)\bigg)
\]
with probability $1 - O((\log x)^{-10})$, and 
Theorem \ref{thm:fgkmt4} (iv) follows on recalling that, by 
\eqref{eq:fgkmt6.11}, 
$
 \sigma y/\log x \sim 80cx \log_2 x /\log x.
$
\end{proof}

\begin{lemma}
 \label{lem:rcb5}
\textup{(}i\textup{)}
With probability 
% $1 - O\big(1/\log x\big)$, 
$1 - O\big(\frac{1}{\log x}\big)$, 
$\cP(\vec{\ba})$ contains all but 
% $O\big(\#\cP/(\log x)^3\big)$ 
$O\big(\frac{\#\cP}{(\log x)^3}\big)$ 
of the primes in $\cP$.
\textup{(}ii\textup{)}
We have 
\[
 \EE \, \#\cP(\vec{\ba}) 
  = (\#\cP)
     \bigg(
      1 + O\bigg(\frac{1}{(\log x)^4}\bigg)
     \bigg).
\]
\end{lemma}
\begin{proof}
We have 
\begin{align*}
 \EE \, X_p(\vec{\ba})
  & = 
   \sum_{n \in \ZZ}
    \PP(\tilde{\bn}_p = n)
     \PP(n + h_i p \in S(\vec{\ba}) \,\, \forall  i \le r)
 \\
 & = 
  \sum_{n \in \ZZ} 
   \PP(\tilde{\bn}_p = p)
    \sigma^r
     \bigg(1 + O\bigg(\frac{1}{(\log x)^{16}}\bigg)\bigg) 
 \\
 & = 
  \sigma^r
   \bigg(1 + O\bigg(\frac{1}{(\log x)^{16}}\bigg)\bigg).
\end{align*}
(For the second step, we supplement Lemma \ref{lem:rcb4} with 
the observation that $\PP(\bn_p = n) = 0$ whenever 
$n + h_ip \in \cH$ for some $i \le r$.) 

Let $\tilde{\bn}_p^{(1)}$ and $\tilde{\bn}_p^{(2)}$ be independent 
random variables having the same probability distribution as 
$\tilde{\bn}_p$.
Then 
\begin{align*}
 X_p(\vec{\ba})^2
  & = 
   \PP\big(\tilde{\bn}_p^{(1)} \in S(\vec{\ba}) \,\, \forall  i \le r\big)
    \PP\big(\tilde{\bn}_p^{(2)} \in S(\vec{\ba}) \,\, \forall  i \le r\big)
  \\
  & = 
   \PP\big(\tilde{\bn}_p^{(l)} \in S(\vec{\ba}) \,\, \forall l \le 2, i \le r\big).
\end{align*}
Arguing as above, 
\[
 \EE X_p(\vec{\ba})^2
  = 
   \sum_{n_1 \in \ZZ}
    \sum_{n_2 \in \ZZ} 
     \PP\big(\tilde{\bn}_p^{(1)} = n_1\big)
      \PP\big(\tilde{\bn}_p^{(2)} = n_2\big)
       \sigma^{t(n_1,n_2)}
         \bigg(1 + O\bigg(\frac{1}{(\log x)^{16}}\bigg)\bigg),
\]
where $t(n_1,n_2)$ is the number of distinct integers 
$n_l + h_ip$ ($l \le 2$, $i \le r$).

Now fix $n_1$.
There are less than $r^2$ values of $n_2$ for which 
$t(n_1,n_2) \ne 2r$.
Since $\PP(\bn_p^{(2)} = n_2) \ll x^{-2/3 + o(1)}$ (cf.\ 
\eqref{eq:fgkmt6.10}), we obtain
\begin{align*}
 &
  \EE X_p(\vec{\ba})^2
 \\
 & \hspace{15pt}
 =
   \sum_{n_1 \in \ZZ}
    \PP(\tilde{\bn}_p^{(1)} = n_1)
     \bigg\{
         \sigma^r 
          \bigg(1 + O\bigg(\frac{1}{(\log x)^{16}}\bigg)\bigg)
           \bigg(1 - O\bigg(\frac{1}{x^{1/2}}\bigg)\bigg)
          + O\bigg(\frac{1}{x^{1/2}}\bigg)
     \bigg\}
 \\
 & \hspace{15pt}
  = 
   \sigma^{2r}
    \bigg(1 + O\bigg(\frac{1}{(\log x)^{16}}\bigg)\bigg).
\end{align*}
Arguing as in \eqref{eq:rcb15}, \eqref{eq:rcb16}, 
\[
 \PP
  \Big(
      |
       X_p(\vec{\ba}) 
        - \big(
           1 + O\big({\textstyle \frac{1}{(\log x)^{16}}}\big)
          \big) 
           \sigma^r
      |
       > 
         {\textstyle 
           \frac{1}{2} 
            \frac{{\displaystyle \sigma^r}}{(\log x)^6}}
  \Big)
   \ll
    \frac{1}{(\log x)^4}.
\]
Thus, with probability 
%$1 - O\big(\frac{1}{(\log x)^4}\big)$, 
$1 - O\big(1/(\log x)^4\big)$, 
we have 
\[
 X_p(\vec{\ba})
  =
   \bigg(
    1 + O_{\le}\bigg(\frac{1}{(\log x)^6}\bigg)
   \bigg) 
    \sigma^r,
\]
that is, $p \in \cP(\vec{\ba})$.
Moreover, 
\begin{align*}
 & 
 \frac{\#\cP}{(\log x)^3}
  \PP
   \Big(
        {\textstyle \sum_{p \in \cP} } 
         \ind{p \not\in \cP(\vec{\ba})}
        >
          {\textstyle\frac{\#\cP}{(\log x)^3}}
   \Big)
 \\
 & \hspace{60pt}
  \le
   \EE \, 
    \Big(
     \sum_{p \in \cP}  
      \ind{p \not\in \cP(\vec{\ba})}
    \Big)
  = 
   \sum_{p \in \cP}
    \PP(p \not\in \cP(\vec{\ba}))
  \ll
   \frac{\#\cP}{(\log x)^4}.
\end{align*}
So with probability $1 - \big(1/\log x\big)$, 
$\cP(\vec{\ba})$ contains all but 
$O\big(\frac{\#\cP}{(\log x)^3}\big)$ of the primes $p \in \cP$.
Finally, 
\[
 \EE \, \#\cP(\vec{\ba})
  = 
   \sum_{p \in \cP}
    \PP(p \in \cP(\vec{\ba}))
  =
   (\#\cP)
     \bigg(1 + O\bigg(\frac{1}{(\log x)^4}\bigg)\bigg).
\]
\end{proof}

\begin{lemma}
 \label{lem:rcb6}
Let $q \in \cQ$ and let $\vec{a}$ be in the essential range of 
$\vec{\ba}$.
Then
\begin{equation}
 \label{eq:rcb17}
  \sigma^{-r}
   \sum_{i=1}^r 
    \sum_{p \, \in \, \cP(\vec{a})}
     Z_p(\vec{a};q - h_ip)
    =
     \big(1 + O\bigg(\frac{1}{(\log x)^6}\bigg)\bigg)
      \sum_{p \, \in \, \cP}
       \PP(q \in \be_p(\vec{a}) \mid \vec{\ba} = \vec{a}).
\end{equation}
\end{lemma}

\begin{proof}
Recalling \eqref{eq:rcb11}, the left-hand side of 
\eqref{eq:rcb17} is 
\begin{align*}
 & 
 \sigma^{-r}
  \sum_{i=1}^r 
   \sum_{p \, \in \, \cP(\vec{a})}
    X_p(\vec{a}) 
     \PP(\bn_p = q - h_ip \mid \vec{\ba} = \vec{a})
 \\
  & \hspace{60pt}
   = 
    \bigg(1 + O\bigg(\frac{1}{(\log x)^6}\bigg)\bigg)
     \sigma^{-r}
      \sum_{i=1}^r 
       \sum_{p \, \in \, \cP(\vec{a})}
        \PP(\bn_p = q - h_ip \mid \vec{\ba} = \vec{a}).  
\end{align*}
Since $q - h_ip \ne \bn_p$ if $p \not\in \cP(\vec{a})$, we may 
rewrite this as 
\[
 \bigg(1 + O\bigg(\frac{1}{(\log x)^6}\bigg)\bigg)
  \sigma^{-r}
   \sum_{i=1}^r 
    \sum_{p \, \in \, \cP}
     \PP(\bn_p = q - h_ip \mid \vec{\ba} = \vec{a}),  
\]
and the lemma follows.
\end{proof}

It is convenient to write 
\[
 U(q,\vec{a})
  \defeq 
  \sigma^{-r}
   \sum_{i=1}^r 
    \sum_{p \, \in \, \cP}
     Z_p(\vec{a};q - h_ip).
\]

\begin{lemma}
 \label{lem:rcb7}
We have
\begin{equation}
 \label{eq:rcb18}
  \EE \,
   \Big(
    \sum_{q \in \cQ \cap S(\vec{\ba})}
     \sigma^r U(q,\vec{\ba})
   \Big)
  =
   \bigg(1 + O\bigg(\frac{1}{(\log_2 x)^{10}}\bigg)\bigg)
    \frac{\sigma y}{\log x}
     \frac{\sigma^{r-1}ux}{2y}
\end{equation}
and
\begin{equation}
 \label{eq:rcb19}
  \EE \,
   \Big(
    \sum_{q \in \cQ \cap S(\vec{\ba})}
     \sigma^r U(q,\vec{\ba})^2
   \Big)
  =
   \bigg(1 + O\bigg(\frac{1}{(\log_2 x)^{10}}\bigg)\bigg)
    \frac{\sigma y}{\log x}
     \bigg(\frac{\sigma^{r-1}ux}{2y}\bigg)^2.
\end{equation}
\end{lemma}

\begin{proof}
We begin with \eqref{eq:rcb19}.
Let $\tilde{\bn}_p^{(1)}$ and $\tilde{\bn}_p^{(2)}$ be 
independent copies of $\tilde{\bn}_p$ that are also independent of 
$\vec{\ba}$.
We observe that for any $n_1,n_2$, and $p_1,p_2 \in \cP$, 
\[
 Z_{p_1}(\vec{\ba};n_1)
  Z_{p_2}(\vec{\ba};n_2)
 =
  \ind{
      (n_l + h_jp_l \, \in \, S(\vec{\ba}) \, \, \forall l \le 2, \, j \le r)
      }
   \,
    \PP(\tilde{\bn}_p^{(1)} = n_1)
     \PP(\tilde{\bn}_p^{(2)} = n_2).
\]
The left-hand side of \eqref{eq:rcb19} is thus
\begin{align}
 \label{eq:rcb20}
  \begin{split}
   & 
   \EE \,
    \Big(
     \sum_{q \in \cQ \cap S(\vec{\ba})}
      \,
       \sum_{i_1,i_2 = 1}^r
        \,
         \sum_{p_1,p_2 \in \cP}
         Z_{p_1}(\vec{\ba};q - h_{i_1}p_1)
          Z_{p_2}(\vec{\ba};q - h_{i_2}p_2)
    \Big)
   \\
    &  
     =
      \sum_{q \in \cQ}
       \,
       \sum_{i_1,i_2 = 1}^r
        \big( 
         \PP(q + (h_j - h_{i_l})p_l \in S(\vec{\ba}) \, \, \forall l \le 2, j \le r)
  \\
   & \hspace{90pt} 
      \times    
       \PP(\tilde{\bn}_{p_1}^{(1)} = q - h_{i_1}p_1)
        \PP(\tilde{\bn}_{p_2}^{(2)} = q - h_{i_2}p_2)
       \big).
  \end{split}
\end{align}
(For $q \in \cQ$, the event 
$q + (h_j - h_{i_l})p_l \in S(\vec{\ba}) \, \, \forall l \le 2, j \le r$ 
is identical to the event 
$q \in \cQ \cap S(\vec{\ba})$, 
$q + (h_j - h_{i_l})p_l \in S(\vec{\ba}) \, \, \forall l \le 2, j \le r$.) 

Let $\Sigma_1,\Sigma_2$ be the contributions to the right-hand 
side of \eqref{eq:rcb20} from $p_1 \ne p_2$, respectively 
$p_1 = p_2$.

Fix $p_1,p_2$ in $\cP$ with $p_1 \ne p_2$, and fix 
$i_1,i_2 \le r$.
The number of distinct integers $q + (h_j - h_{i_l})p_l$ 
($l \le 2, j \le r$) is $2r - 1$ since 
$
 (h_j - h_{i_1})p_1 \ne (h_j - h_{i_2})p_2
$
for $i_1 \ne j$.
Hence 
\begin{equation}
 \label{eq:rcb21}
  \begin{split}
   & 
   \PP(q + (h_j - h_{i_l})p_l \in S(\vec{\ba}) \, \, \forall l \le 2, j \le r)
    \PP(\tilde{\bn}_p^{(1)} = q - h_{i_1}p_1)
     \PP(\tilde{\bn}_p^{(2)} = q - h_{i_2}p_2)
   \\ 
   & \hspace{30pt}
   =
    \bigg(1 + O\bigg(\frac{1}{(\log x)^{16}}\bigg)\bigg) 
     \sigma^{2r-1}
      \PP(\tilde{\bn}_p^{(1)} = q - h_{i_1}p_1)
       \PP(\tilde{\bn}_p^{(2)} = q - h_{i_2}p_2).      
  \end{split}
\end{equation}
(Both sides are $0$ if there are $j,i_l$ with 
$q + (h_j - h_{i_l})p_l \in \cH$; otherwise 
Lemma \ref{lem:rcb4} applies.)

We combine \eqref{eq:rcb21} with the remark after 
\eqref{eq:probnph0} to obtain 
\begin{align}
 \label{eq:rcb22}
  \begin{split}
   \sumsstxt[][][1]
    & 
     =
      \sums[q \in \cQ \setminus \cH][i_1,i_2 \le r]
       \sigma^{2r-1}
        \bigg(1 + O\bigg(\frac{1}{(\log_2 x)^{10}}\bigg)\bigg)
         \bigg(\frac{ux}{2ry}\bigg)^2
   \\
   & = 
    \bigg(1 + O\bigg(\frac{1}{(\log_2 x)^{10}}\bigg)\bigg)
     \frac{\sigma y}{\log x}
      \bigg(\frac{\sigma^{r-1}ux}{2y}\bigg)^2.
  \end{split}
\end{align}
Similarly, 
\begin{align*}
 \sumsstxt[][][2]
  & 
   =
    \sum_{q \in \cQ}
     \sum_{i_1 = 1}^r
      \sum_{p_1 \in \cP}
       \PP(q + (h_j - h_{i_1})p_1 \in S(\vec{\ba}) \, \, \forall j \le r)
        \PP(\tilde{\bn}_p^{(1)} = q - h_{i_1}p_1)^2
  \\
   & 
    \ll
     x^{-3/5}
      \sum_{q \in \cQ}
       \sum_{i_1 = 1}^r 
        \sum_{p_1 \in \cP}
         \PP(\tilde{\bn}_p^{(1)} = q - h_{i_1}p_1)
  \\
   & 
    \ll
     x^{-3/5}(\#\cQ),
\end{align*}
which together with \eqref{eq:rcb22} yields \eqref{eq:rcb19}.

Much the same argument gives for the left-hand side of 
\eqref{eq:rcb18} the expression 
\begin{align*}
 & 
 \sum_{q \in \cQ}
  \sum_{i = 1}^r
   \sum_{p \in \cP}
    \PP(q + (h_j - h_i)p \in S(\vec{\ba}) \, \, \forall j \le r)
     \PP(\tilde{\bn}_p = q - h_ip)
 \\
 & \hspace{30pt}
  = 
   \sum_{q \in \cQ}
    \sum_{i=1}^r
     \sum_{p \in \cP}
      \bigg(1 + O\bigg(\frac{1}{(\log x)^{16}}\bigg)\bigg)
       \sigma^r
        \PP(\tilde{\bn}_p = q - h_ip)
 \\
 & \hspace{30pt}
  =
    \bigg(1 + O\bigg(\frac{1}{(\log_2 x)^{10}}\bigg)\bigg)
     \frac{\sigma y}{\log x}
      \frac{\sigma^{r-1}ux}{2y}.
\end{align*}
\end{proof}

We now specify that the quantity $C$ in Theorem \ref{thm:fgkmt4} 
is 
\[
 C \defeq \frac{ux}{2\sigma y},
\]
so that $C \asymp 1/c$.

\begin{lemma}
 \label{lem:rcb8}
With probability $1 - o(1)$, we have 
\[
 U(q,\vec{\ba}) 
  = 
   \bigg(
    1 + O_{\le}\bigg(\frac{1}{(\log_2 x)^3}\bigg)
   \bigg)
    C
\]
for all but at most $\frac{x}{2\log x\log_2 x}$ of the primes 
$q \in \cQ \cap S(\vec{\ba})$.
\end{lemma}

\begin{proof}
Using \eqref{eq:rcb14} and Lemma \ref{lem:rcb7}, we find that 
\begin{align} 
 \label{eq:rcb23}
  \begin{split}
 & 
 \EE \,
  \Big(
   \sum_{q \in \cQ \cap S(\vec{\ba})}
    (U(q,\vec{\ba}) - C)^2
   \Big)
 \\
 & 
  =
   \EE \,
    \Big(
     \sum_{q \in \cQ \cap S(\vec{\ba})} U(q,\vec{\ba})^2    
    \Big)
    -
      2C \,
       \EE \, 
        \Big(
         \sum_{q \in \cQ \cap S(\vec{\ba})} U(q,\vec{\ba})
        \Big)
       + 
         C^2 \,
          \EE \,
           \Big(
            \sum_{q \in \cQ \cap S(\vec{\ba})} 1
           \Big) 
  \\
  & 
   =
    \bigg(1 + O\bigg(\frac{1}{(\log_2 x)^{10}}\bigg)\bigg)
     \frac{u^2x^2}{4\sigma y\log x}
 \\
  & \hspace{45pt}    
    -
      \bigg(1 + O\bigg(\frac{1}{(\log_2 x)^{10}}\bigg)\bigg)
       \frac{2ux}{2\sigma y}
        \frac{ux}{2\log x}
   +
    \big(1 + O\big({\textstyle \frac{1}{\log x}}\big)\big)
     \frac{u^2x^2}{4\sigma^2y^2}
      \frac{\sigma y}{\log x}
 \\
  & 
   \ll
    \frac{u^2x^2}{\sigma y(\log x)(\log_2 x)^{10}}
 \\
 & 
  \ll
   \frac{C^2\sigma y}{(\log x)(\log_2 x)^{10}}.
  \end{split}
\end{align}
Let $V$ be the event 
\[
 \#\big\{
     q \in \cQ \cap S(\vec{\ba}) : 
      |U(q,\vec{\ba}) - C| 
       >
        {\textstyle \frac{C}{(\log_2 x)^3}}
    \big\}
    >
     \frac{x}{2\log x \log_2 x}.
\]
Evidently 
\[
 \EE \,
  \Big(
   \sum_{q \in \cQ \cap S(\vec{\ba})}
    (U(q,\vec{\ba}) - C)^2
  \Big)
   \ge 
    \PP(V)
     \frac{x}{2\log x\log_2 x}
      \frac{C^2}{(\log_2 x)^6}.
\]
Combining this with \eqref{eq:rcb23}, and recalling that 
$\sigma y \asymp x\log_2 x$ (cf.\ \eqref{eq:fgkmt6.11}), we obtain 
\[
 \PP(V) \ll \frac{1}{(\log_2 x)^2}.
\]
\end{proof}

\begin{lemma}
 \label{lem:rcb9}
We have 
\begin{equation}
 \label{eq:rcb24}
  \EE \, 
   \Big(
    \sum_{n \in \ZZ}
     \sigma^{-r}
      \sum_{p \in \cP(\vec{\ba})} 
       Z_p(\vec{\ba};n)
   \Big)
  =
   \bigg(1 + O\bigg(\frac{1}{(\log x)^4}\bigg)\bigg)
    (\#\cP)
\end{equation}
and 
\begin{equation}
 \label{eq:rcb25}
  \EE \, 
   \Big(
    \sum_{n \in \ZZ}
     \sigma^{-r}
      \sum_{p \in \cP \setminus \cP(\vec{\ba})} 
       Z_p(\vec{\ba};n)
   \Big)
  \ll
   \frac{\#\cP}{(\log x)^4}.
\end{equation}
\end{lemma}

\begin{proof}
The left-hand side of \eqref{eq:rcb24} is 
\begin{align*}
 & 
  \sigma^{-r}
   \sum_{\vec{\ba}}
    \PP(\vec{\ba} = \vec{a})
     \sum_{p \in \cP(\vec{a})}
      \sum_{n \in \ZZ}
       \PP(n + h_ip \in S(\vec{a}) \, \, \forall i \le r)
        \PP(\tilde{\bn}_p = n)
 \\
  & \hspace{30pt} = 
   \sigma^{-r}
    \sum_{\vec{\ba}}
     \PP(\vec{\ba} = \vec{a})
      \sum_{p \in \cP(\vec{a})}
       X_p(\vec{a})
 \\
 & \hspace{30pt}  = 
  \bigg(1 + O\bigg(\frac{1}{(\log x)^6}\bigg)\bigg)
   \sum_{\vec{\ba}}
    \PP(\vec{\ba} = \vec{a})
     \sum_{p \in \cP(\vec{a})} 1
 \\
 & \hspace{30pt}  = 
  \bigg(1 + O\bigg(\frac{1}{(\log x)^6}\bigg)\bigg)
   \EE \, \#\cP(\vec{a})
 \\
 & \hspace{30pt}  = 
  \bigg(1 + O\bigg(\frac{1}{(\log x)^4}\bigg)\bigg)
   (\# \cP)
\end{align*}
by Lemma \ref{lem:rcb5}.
This proves \eqref{eq:rcb24}.

Now, 
\begin{align}
 \label{eq:rcb26}
  \begin{split}
   \EE \,
    \Big(
     \sum_{n \in \ZZ} \sigma^{-r}
      \sum_{p \in \cP} Z_p(\vec{\ba};n)
    \Big)
  & 
   =
    \sigma^{-r}
     \sum_{p \in \cP}
      \sum_{n \in \ZZ}
       \PP(\tilde{\bn}_p = n)
        \PP(n + h_jp \in S(\vec{a}) \,\, \forall j \le r)
 \\
 & 
  =
   \bigg(1 + O\bigg(\frac{1}{(\log x)^{16}}\bigg)\bigg)
    \sum_{p \in \cP}
     \sum_{n \in \ZZ} 
      \PP(\tilde{\bn}_p = n)
 \\
 & 
  =
   \bigg(1 + O\bigg(\frac{1}{(\log x)^{16}}\bigg)\bigg)
    (\# \cP).
  \end{split}
\end{align}
(Here we have used Lemma \ref{lem:rcb4} with a familiar argument.)
We obtain \eqref{eq:rcb25} on subtracting \eqref{eq:rcb24} from 
\eqref{eq:rcb26}.
\end{proof}

\begin{proof}%
[Deduction of Theorem \ref{thm:fgkmt4} \textup{(}v\textup{)}]
In view of Lemmas \ref{lem:rcb6} and \ref{lem:rcb8} it suffices to 
show that with probability $1 - O\big(1/(\log x)^3\big)$, the 
number of $q$ in $\cQ \cap S(\vec{\ba})$ with 
\begin{equation}
 \label{eq:rcb27}
  \sum_{i = 1}^r
   \sum_{p \in \cP \setminus \cP(\vec{\ba})}
    \sigma^{-r} Z_p(a; q - h_ip)
     >
      \frac{ux}{\sigma y(\log_2 x)^3}
\end{equation}
is at most $\frac{x}{2\log x\log_2 x}$.

Let $W$ be the event that \eqref{eq:rcb27} holds for more than 
$\frac{x}{2\log x\log_2 x}$ primes in 
\linebreak 
$\cQ \cap S(\vec{\ba})$.
Then 
\[
 \PP(W)
  \le 
   \PP
    \Big(
     \sum_{q \in \cQ}
      \sum_{i=1}^r
       \sum_{p \in \cP\setminus \cP(\vec{\ba})}
        \sigma^{-r}
         Z_p(\vec{\ba};q - h_ip) > v
    \Big),
\]
where 
\[
 v 
  \defeq 
   \frac{ux}{\sigma y(\log_2 x)^3}
    \cdot 
     \frac{x}{2\log x \log_2 x}
   =
    \frac{x}{(\log x)^{2 - o(1)}}.
\]
Thus, 
\begin{multline*}
  \PP(W) 
   \le 
    \frac{1}{v} 
     \EE \, 
      \Big(
       \sum_{q \in \cQ}
        \sum_{i=1}^r 
         \sum_{p \in \cP\setminus \cP(\vec{\ba})}
          \sigma^{-r}
           Z_p(q - h_ip)
      \Big)
 \\
  \le 
   \frac{r}{v}
    \EE \,
     \Big(
      \sum_{n \in \ZZ}
       \sigma^{-r} 
        \sum_{p \in \cP\setminus \cP(\vec{\ba})}
         Z_p(\vec{\ba};n)
     \Big)
  \ll
   \frac{r}{v}
    \frac{x}{(\log x)^5}
  \ll
   \frac{1}{(\log x)^2}.
\end{multline*}
\end{proof}

\subsection{Proof of Theorem \ref{thm:fgkmt2}}
 \label{subsec:thm2pf}
We require one further lemma for the proof of 
Theorem \ref{thm:fgkmt2}, viz.\ the following, which is a special 
case of \cite[Corollary 3]{FGKMT}.

\begin{lemma}
 \label{lem:rcb10}
Let $\cQ'$ be a set of primes with $\#\cQ' > (\log_2 x)^3$.
For each $p \in \cP$, let $\be_p$ be a random subset of $\cQ'$ 
with 
\[
 \# \be_p \le r, 
  \quad 
   \PP(q \in \be_p) \le x^{-3/5} \quad (q \in \cQ').
\]
Suppose that for all but at most $\frac{\#\cQ'}{(\log_2 x)^2}$ 
elements $q \in \cQ'$, we have 
\[
 \sum_{p \in \cP} \PP(q \in \be_p)
  =
   C + O_{\le}\bigg(\frac{1}{(\log_2 x)^2}\bigg),
\]
where $C$ is independent of $q$ and 
\begin{equation}
 \label{eq:rcb28}
  {\textstyle \frac{5}{4}}\log 5 \le C \ll 1.
\end{equation}
Suppose that for any distinct $q_1,q_2 \in \cQ'$, 
\begin{equation}
 \label{eq:rcb29}
  \sum_{p \in \cP'} \PP(q_1,q_2 \in \be_p)
   \le 
    x^{-1/20}.
\end{equation}
Then for any positive integer $m$ with 
\[
 m \le \frac{\log_3 x}{\log 5},
\]
we can find random sets $\be_p' \subseteq \cQ'$ for each 
$p \in \cP$ such that $\be_p'$ is either empty or is in the 
essential range of $\be_p$, and 
\begin{equation}
 \label{eq:rcb30}
  \#\{q \in \cQ' : q \not\in \be_p' \,\, \textup{for all} \,\, p \in \cP \}
   \sim 
    5^{-m}(\#\cQ'),
\end{equation}
with probability $1 - o(1)$.
\end{lemma}

\begin{proof}[Deduction of Theorem \ref{thm:fgkmt2}]
By \eqref{eq:fgkmt4.29}, we may choose $c$ small enough so that 
\eqref{eq:rcb28} holds.
Take 
\[
 m = \Big\lfloor \frac{\log_3 x}{\log 5} \Big\rfloor.
\]
Let $\vec{\ba}$ and $\vec{\bn}$ be as in Theorem \ref{thm:fgkmt4}.
Suppose that we are in the probability $1 - o(1)$ event that 
$\vec{\ba}$ takes a value $\vec{a}$ for which \eqref{eq:fgkmt4.31} 
holds.
Fix some $\vec{a}$ within this event.
We apply Lemma \ref{lem:rcb10} with $\cQ' = \cQ \cap S(\vec{a})$, 
$\be_p = \be_p(\vec{a})$.
We need only check the hypothesis \eqref{eq:rcb29}.
We have 
\[
 \sum_{p \in \cP}
  \PP(q_1,q_2 \in e_p(\vec{a})
   \le 
    \sums[p \mid q_1 - q_2][p \in \cP]
     \PP(q_1 \in e_p(\vec{a}))
      \le 
       x^{-3/5}
\]
(the sum has at most one term).

Let $\be_p'(\vec{a})$ be the random variables provided by Lemma 
\ref{lem:rcb10}.
Recalling \eqref{eq:fgkmt4.30}, 
\[
 \#\{q \in \cQ' : q \not\in \be_p' \,\, \text{for all} \,\, p \in \cP \}
  \sim 
   5^{-m}\#(\cQ \cap S(\vec{a}))
    \ll
     \frac{x}{\log x}
\]
with probability $1 - o(1)$.
Since $e_p'(\vec{a})$ is either empty or 
\[
 e_p'(\vec{a})
  =
   \{\tilde{\bn}_p' + h_ip : i \le r\}
    \cap 
     \cQ \cap S(\vec{a})
\]
for some random integer $\tilde{\bn}_p'$, it follows that 
\[
 \#\{
     q \in \cQ \cap S(\vec{a}) : 
      q \not\equiv \tilde{\bn}_p' \pod{p} 
         \,\, \text{for all} \,\, p \in \cP
   \}
    \ll
     \frac{x}{\log x}
\]
with probability $1 - o(1)$.
The bound \eqref{eq:fgkmt3.6} follows on setting $b_p = n_p'$ for 
a specific $\vec{n}' = (\tilde{\bn}_p')$ for which this bound 
holds.
That $\cH$ is contained in $S(\vec{a}) \cap S(\vec{b})$ follows 
from parts (i) and (ii) of Theorem \ref{thm:fgkmt4}.
\end{proof}

%%%%%%%%%%%%%%%%%%%%%%%%%%%%%%%%%%%%%%%%%%%%%%%%%%%%%%%%%%%%%%%%%%
%%%%%%%%%%%%%%%%%%%%%%%%%%%  SECTION 5 %%%%%%%%%%%%%%%%%%%%%%%%%%%
%%%%%%%%%%%%%%%%%%%%%%%%%%%%%%%%%%%%%%%%%%%%%%%%%%%%%%%%%%%%%%%%%%

\section{A modification of Maynard--Tao}
 \label{sec:MT}
 
\begin{definition} 
 \label{def:w}
We consider functions of the form 
$f_1 : [T_1,\infty) \to [x_1,\infty)$, with $T_1,x_1 \ge 1$.
Let us say that such a function $f_1$ is ``of the first kind'' if 
and only if 
(i) it is a strictly increasing bijection, 
(ii) $f_1(T) \le \log T$ for $T \ge T_1$,
(iii) $f_1(2T)/f_1(T) \to 1$ as $T \to \infty$
and
(iv) for $0 < \eta \le 1$, there exists $L_{\eta} \ge 1$ such that 
$f_1(T)/f_1(T^{\eta}) \to L_{\eta}$ as 
$T \to \infty$.
\end{definition}

\begin{definition}
 \label{def:sparse}
We consider (possibly empty) sets $\cZ(T)$, $T \ge 2$, of primes 
less than or equal to $T$.
Let us that such a set is ``repulsive'' if and only if for any 
$p' \in \cZ(T)$,  
$\sum_{p \in \cZ(T), \, p \ge p'} 1/p \ll 1/p' \ll 1/\log_2 T$.
\end{definition}

Given a function $\upsilon : \NN \to \RR$ with finite support and 
any arithmetic progression $a \pod{D}$ with $(a,D) = 1$, we 
define  
\[
 \Delta(\upsilon;a \pod{D})
  \defeq 
    \sum_{n \equiv a \pod{D}}
     \upsilon(n)
      -
       \frac{1}{\phi(D)}
        \sum_{(n,D) = 1}
         \upsilon(n),
\]
where $\phi$ is Euler's totient function.

\begin{hypothesis} 
 \label{hyp:EH}
Fix $\theta \in (0,1]$ and a function 
$f_1 : [T_1,\infty) \to [x_1,\infty)$ of the first kind.
For any given $A > 0$ and $\delta \in (0,\theta)$, if 
$\eta = \eta(A,\delta) \in (0,\theta - \delta)$ is a 
sufficiently small, fixed number then, for $N \ge T_1^{1/\eta}$, 
there is a repulsive subset $\cZ \defeq \cZ(N^{4\eta})$ of the 
primes less than or equal to $N^{4\eta}$ such that, with 
$W \defeq \prod_{p \le N^{\eta}, \, p \not\in \cZ} p$ and 
$Z \defeq \prod_{p \in \cZ} p$, we have
\[
  \sums[r \le N^{\theta}/(N^{\delta}W)]
       [(r,WZ) = 1]
       [\text{$r$ squarefree}]
     \max_{N \le M \le 2N}
      \max_{(rW,a) = 1}     
      |\Delta(\ind{\bP}\ind{(M,M + N]}; a \pod{rW})|
      \ll_{\delta,A}
       \frac{N}{\phi(W)(\log N)^A}.
\]
\end{hypothesis}

\begin{theorem}
 \label{thm:BFM4.3}
Fix $\theta \in (0,1]$ and a function 
$f_1 : [T_1,\infty) \to [x_1,\infty)$ of the first kind.
Suppose that Hypothesis \ref{hyp:EH} holds.
Fix a positive integer $a$.
In the notation of Hypothesis \ref{hyp:EH}, if 
$K = K_{\theta,a}$ is a sufficiently large integer multiple of 
$\lceil (2/\theta) a \rceil + 1$, if $A = A_K$ is sufficiently 
large and if $\delta = \delta_{\theta,a}$ is sufficiently small, 
then the following holds for $N \ge N(T_1,K,\eta)$.
Let $\cH \defeq \{H_1,\ldots,H_K\} \subseteq [0,N]$ be an 
admissible set of $K$ distinct integers for which 
$\prod_{1 \le i < j \le K}(H_j - H_i)$ is 
$f_1(N^{\eta})$-smooth, and let $b$ be an integer such that  
\[
 \textstyle (\prod_{i = 1}^K (b + H_i),W) = 1.
\]
Then for any partition 
\[
 \cH  
  = 
   \cH_1 \cup \cdots \cup \cH_{\lceil (2/\theta) a \rceil + 1}   
\]
of $\cH$ into $\lceil (2/\theta) a \rceil + 1$ sets of equal 
size, there exists some $n \in (N,2N] \cap b \pod{W}$, and $a + 1$ 
distinct indices 
$
 i_1,\ldots,i_{a+1} 
  \in 
   \{1,\ldots,\lceil (2/\theta) a \rceil + 1\}
$, 
such that 
\[
 \#(\bP \cap n + \cH_{i_1}),
  \ldots,
   \#(\bP \cap n + \cH_{i_{a + 1}}) \ge 1.
\]
\end{theorem}

\begin{theorem}
 \label{thm:BFM4.2}
Hypothesis \ref{hyp:EH}, and therefore the statement of 
Theorem \ref{thm:BFM4.3}, holds with $\theta = 1/2$ and any 
function $f_1 : [T_1,\infty) \to [x_1,\infty)$ of the first kind.
\end{theorem}

\begin{proof}%
[Proof of Theorems \ref{thm:BFM4.2} and \ref{thm:BFM4.3}]
That Hypothesis \ref{hyp:EH} holds with $\theta = 1/2$ and any 
function $f_1 : [T_1,\infty) \to [x_1,\infty)$ of the first kind 
is a consequence of Lemma 4.1 and Theorem 4.2 of \cite{BFM}.

We prove Theorem \ref{thm:BFM4.3} by following Pintz's \cite{PIN3} 
modification to the proof of Theorem 4.3 (i) in \cite{BFM}.
There are many parameters involved and it is important to keep 
track of their interdependencies.
It is also important to note that the implicit constants in all 
$O$-terms are absolute, that is, independent of all 
parameters.

Only the unconditional case $\theta = 1/2$ is considered in 
\cite{BFM,PIN3}, whereas here we are considering $\theta \le 1$.
To do this, we need to note that on Hypothesis \ref{hyp:EH}, the 
term $4 + O(\delta)$ may be replaced by $(2/\theta) + O(\delta)$ 
on the right-hand side of the inequality in 
\cite[Lemma 4.5 (iii)]{BFM}.
In the proof of this lemma in \cite[\S 4.2]{BFM}, the support of 
the smooth function $G : [0,\infty) \to \RR$, which is 
$[0,1/4 - 2\delta]$, may be replaced by 
$[0,(\theta/2) - 2\delta]$, and the rest of the proof may be 
carried out, mutatis mutandis.

As in \cite{PIN3}, we begin with the following observation.
Suppose $K$ and $M$ are positive integers with $M \mid K$, and 
let $\cH = \cH_1 \cup \cdots \cup \cH_M$ be a partition of a set 
$\cH$ of integers into $M$ subsets of equal size.
Suppose also that $\mu'$ and $\mu$ are positive real numbers 
with 
\[
 \mu' 
  \defeq
   \max_{v \in \NN}
    \bigg(v - \mu \binom{v}{2}\bigg).
\]

Given an integer $n$, consider the expression  
\[
 \sum_{j=1}^M
  \Big\{ 
   \sum_{H \in \cH_j} \ind{\bP}(n + H)
    -
     \mu 
      \sums[H,H' \in \cH_j][H \ne H']
       \ind{\bP}(n + H)\ind{\bP}(n + H')   
   \Big\},
\]
where in the double sum each unordered pair 
$\{H,H'\} \subseteq \cH_j$ with $H \ne H'$ is counted once only. 
Suppose $\#(\bP \cap n + \cH_j) = 0$ for all but at most $a$ of 
the subsets $\cH_j$.
Then the above expression is at most $\mu' a$.
Consequently, if 
\[
 \sum_{H \in \cH}
  \ind{\bP}(n + H)
   -
    \mu' a
     - 
      \mu
       \sum_{j=1}^M
        \sums[H,H' \in \cH_j][H \ne H']
         \ind{\bP}(n + H)
          \ind{\bP}(n + H')
\]
is positive then $\#(\bP \cap n + \cH_j) \ge 1$ for at least 
$a + 1$ of the subsets $\cH_j$.

Note that when $\mu$ is the reciprocal of a positive integer, 
we have 
\[
 \mu' 
  = 
    \textstyle 
     \frac{1}{2}\big(1 + \frac{1}{\mu}\big),
\]
the maximum being attained when $v = 1/\mu$ and 
$v = 1 + 1/\mu$.

Now, fix $\theta \in (0,1]$ and a function 
$f_1 : [T_1,\infty) \to [x_1,\infty)$ of the first kind.
Suppose that Hypothesis \ref{hyp:EH} holds.
Fix any positive integer $a$ and let $M = M_{\theta,a}$ be 
the integer satisfying 
\[
 M - 2 < (2/\theta)a \le M - 1.
\]
Let $\iota = \iota_{\theta,a}$ be a small, fixed quantity to 
be specified.
Set 
\[
 \delta = \delta_{\theta,a} \defeq \iota^2/(aM).
\]
Let $K = K_{\theta,a}$ be the integer satisfying
\[
 \e^{aM^2/(\delta(M-1))} < K \le \e^{aM^2/(\delta(M-1))} + M
  \quad 
   \text{and}
    \quad 
     M \mid K.
\]
Finally, let
\[
 \rho \defeq \frac{aM^2/(M-1)}{\delta\log K} < 1.
\]

Now let $A = A(K)$, $\eta = \eta(A,\delta)$, 
$N \ge N(T_1,K,\eta)$, $\cH = \{H_1,\ldots,H_K\}$ and $b \pod{W}$ 
be as in the statement of the theorem.
Let $\cH = \cH_1 \cup \cdots \cup \cH_M$ be any partition of $\cH$ 
into $M$ subsets of equal size.
Consider the expression
\[
 S 
  \defeq 
  \hspace{-7pt}  
   \sums[N < n \le 2N][n \equiv b \pod{W}]
    \hspace{-3pt}
    \Big\{
     \sum_{H \in \cH} \ind{\bP}(n + H)
     -
      \frac{1 + M}{2}a
       -
         \frac{1}{M}
          \sum_{j=1}^M
           \sums[H,H' \in \cH_j][H \ne H']
            \hspace{-5pt}
             \ind{\bP}(n + H)\ind{\bP}(n + H')
    \Big\}
     \nu_{\cH}(n),
\]
where $\nu_{\cH} : \NN \to [0,\infty)$ is the nonnegative weight 
given by 
\[
 \nu_{\cH}(n)
  \defeq 
   \Big( 
    \sums[d_1,\ldots,d_K][d_i \mid n + H_i \,\, \forall i \le K]
     \lambda_{d_1,\ldots,d_K}
   \Big)^2,
\]
and where $(\lambda_{d_1,\ldots,d_K})$ is the Maynard--Tao sieve 
as used in \cite[\S 4]{BFM}.
The aim is to show that $S > 0$, for in that case, by the 
observation made at the beginning of the proof, there must exist 
some $n \in (N,2N] \cap b \pod{W}$ and $a + 1$ subsets $\cH_j$ 
for which $\#(\bP \cap n + \cH_{j}) \ge 1$.

At this point we invoke Lemmas 4.5 and 4.6 in \cite{BFM} (with 
$4 + O(\delta)$ in the latter replaced by 
$2/\theta + O(\delta) \le (M - 1)/a + O(\delta)$).
To ease notation define $\mathfrak{S}$ by the relation 
$
 S = \mathfrak{S}NW^{-1}B^{-K}I_K(F),
$
with $B$ and $I_K(F)$ as defined in \cite[\S 4.2]{BFM}.
Also let $\xi = (\log K)^{-1/2}$.

As in \cite[\S 4.2]{BFM} and \cite[(3.13)]{PIN3}, we find that the 
relevant estimates yield 
\begin{align*}
 \mathfrak{S}
  & 
   \ge
    \sum_{H \in \cH} \frac{aM^2/(M-1)}{K}(1 + O(\xi))
   - \frac{1 + M}{2}a 
  \\
  & \hspace{30pt}      
       - \frac{1}{M}
          \sum_{j=1}^M 
           \sums[H,H' \in \cH_j][H \ne H']
            \frac{M-1}{a}\cdot 
             \frac{(aM^2/(M-1))^2}{K^2}
             (1 + O(\delta + \xi))
 \\
  & 
   =
    \frac{aM^2}{M - 1}(1 + O(\xi))
     - 
      \frac{1 + M}{2}a 
       -
          \binom{K/M}{2}
           \frac{aM^4(1 + O(\delta + \xi))}{K^2(M-1)}
            .
\end{align*}
Recalling that $\e^{aM^2/(\delta(M-1))} < K$, we see that 
$M/K < \delta$ and hence 
\[
 \binom{K/M}{2}
  = 
   \frac{K^2}{2M^2}\Big(1 - \frac{M}{K}\Big)
    =
       \frac{K^2}{2M^2}(1 + O(\delta)).
\]
We also have $\xi^2 < \delta/(aM) = \iota^2/(a^2M^2)$, so 
$\delta + \xi = O(\iota/(aM))$.
We therefore have
\begin{align*}
 \mathfrak{S}
  & \ge 
   \frac{aM^2}{M-1}\Big(1 + O\Big(\frac{\iota}{aM}\Big)\Big) 
  - \frac{1 + M}{2}a 
   - \frac{aM^2}{2(M-1)}\Big(1 + O\Big(\frac{\iota}{aM}\Big)\Big)
  \\
  & =
    \frac{a(1 + O(\iota))}{2(M - 1)}.
\end{align*}
Taking $\iota$ sufficiently small gives $\mathfrak{S} > 0$, and 
hence $S > 0$, as desired.
\end{proof}

%%%%%%%%%%%%%%%%%%%%%%%%%%%%%%%%%%%%%%%%%%%%%%%%%%%%%%%%%%%%%%%%%%
%%%%%%%%%%%%%%%%%%%%%%%%%%%  SECTION 6 %%%%%%%%%%%%%%%%%%%%%%%%%%%
%%%%%%%%%%%%%%%%%%%%%%%%%%%%%%%%%%%%%%%%%%%%%%%%%%%%%%%%%%%%%%%%%%

\section{Main Theorem and Deduction of Theorem \ref{thm:main}}
 \label{sec:BFM}

Recall Definition \ref{def:w}, in which functions ``of the first 
kind'' are introduced.
We now define a second kind of function.

\begin{definition} 
 \label{def:2ndkind}
We consider functions $f_2 : [x_2,\infty) \to [z_2,\infty)$, with 
$x_2,z_2 \ge 1$. \linebreak
Let us say that such a function $f_2$ is ``of the second kind'' 
if and only if
(i) it is a strictly increasing bijection, 
(ii) 
\[
 (x/\log x)/f_2(x) \to 0
  \quad 
   \text{and} 
    \quad 
   f_2(x)/(x\log x \log_3 x/\log_2 x) \to 0
\]
as $x \to \infty$ and 
(iii) for any $C > 0$, $f_2(Cx)/(Cf(x)) \to 1$ as $x \to \infty$.
\end{definition}

\begin{theorem}
 \label{thm:general}
Fix $\theta \in (0,1]$ and a function 
$f_1 : [T_1,\infty) \to [x_1,\infty)$ 
of the first kind, and suppose Hypothesis \ref{hyp:EH} holds.
Fix a function  
$f_2 : [x_1,\infty) \to [z_2,\infty)$ 
of the second kind and let 
$f \defeq f_2 \circ f_1$.
Let $d_n \defeq p_{n+1} - p_n$, where $p_n$ denotes the $n$th 
smallest prime, and let $\LP[f]$ denote the set of limit 
points in $[0,\infty]$ of the sequence 
$(d_n/f(p_n))_{p_n \ge T_1}$.
Then given any $\lceil (2/\theta)  \rceil + 1$ nonnegative real numbers 
$
\alpha_1,\ldots,\alpha_{\lceil (2/\theta)  \rceil + 1}
$ 
with 
\[
 \alpha_1 \le \cdots \le \alpha_{\lceil (2/\theta)  \rceil + 1},
\] 
we have 
\begin{equation}
 \label{eq:genthm1}
 \{\alpha_j - \alpha_i : 1 \le i < j \le \lceil (2/\theta)  \rceil + 1\}  
  \cap 
   \LP[f]
    \ne 
     \emptyset.
\end{equation}
Consequently, letting $\lambda$ denote the Lebesgue measure on 
$\RR$, we have 
\begin{equation}
 \label{eq:genthm2}
 \lambda([0,X] \cap \LP[f])
  \ge 
   c_1(\theta)X
    \quad (X \ge 0)
\end{equation}
and 
\begin{equation}
 \label{eq:genthm3}
 \lambda([0,X] \cap \LP[f])
  \ge 
   (1 - o(1))
   c_2(\theta)X
    \quad (X \to \infty),
\end{equation}
where 
\begin{equation}
 \label{eq:genthm4}
 c_1(\theta)
  \defeq
   (\lceil (2/\theta)  \rceil(1 + 1/2 + \cdots + 1/\lceil (2/\theta)  \rceil))^{-1}
    \quad
     \text{and}
      \quad  
     c_2(\theta) \defeq 1/\lceil (2/\theta)  \rceil.
\end{equation}
\end{theorem}

\vspace*{1em}

\begin{center}
 \label{tab:1}
\begin{tabular}{|c|c|c|c|} 
 \hline 
 $\theta$                           & $\lceil (2/\theta)  \rceil + 1$     & $c_1(\theta)$    & $c_2(\theta)$  \\ \hline \hline
 $1/2 \le \theta < 2/3$             & $5$                                 & $3/25$           & $1/4$          \\ \hline 
 $2/3 \le \theta < 1\phantom{/1}$   & $4$                                 & $2/11$           & $1/3$          \\ \hline 
 $        \theta = 1$               & $3$                                 & $1/3\phantom{0}$ & $1/2$          \\ \hline 
\end{tabular}
 \vspace*{1em}
 \captionof{table}{Possible values of $\lceil (2/\theta)  \rceil + 1$, $c_1(\theta)$ and $c_2(\theta)$.}
\end{center}

\begin{proof}[Deduction of Theorem \ref{thm:main}]
In view of Theorem \ref{thm:BFM4.2}, we may unconditionally apply 
Theorem \ref{thm:general} with $\theta = 1/2$ and any function 
$
f_1 : [T_1,\infty) \to [x_1,\infty)
$ 
of the first kind.
Let 
$
f_1 : [\e^{\e^{\e}},\infty) \to [\e^{\e},\infty)
$
be given by $f_1(T) = \log T$, and let  
$
f_2 : [\e^{\e},\infty) \to [\e^{\e + 1},\infty)
$
be given by $f_2(x) = x\log x/\log_2 x$.
Then $f_1$ is of the first kind, $f_2$ is of the 
second kind and 
$
 f_2 \circ f_1(T) 
 =
  R_1(T)
   =
    \log T \log_2 T/\log_3 T
$
for $T \ge \e^{\e^{\e}}$.
\end{proof}

\begin{lemma}
 \label{lem:goodktuple}
Let $K$ be a natural number and let $K = K_1 + \cdots + K_M$ be a 
partition of $K$. 
Let $x$ and $y$ be real numbers such that $K \le y/x \le \log x$.
If $x$ is sufficiently large, then for any $M$ 
\textup{(}possibly overlapping\textup{)} subintervals 
$(v_i,v_i + x/\log x] \subseteq (x,y]$, $i \le M$, there exist $M$ 
pairwise disjoint sets of primes  
$\cH_i \subseteq (v_i,v_i + x/\log x]$ with $|\cH_i| = K_i$, such 
that if $\cH_1 \cup \cdots \cup \cH_M = \{q_1,\ldots,q_K\}$, then 
$\prod_{1 \le i < j \le K}(q_j - q_i)$ is $x$-smooth.
\end{lemma}

\begin{proof}
For any $M$ sets $\cJ_i$ with $|\cJ_i| \ge K$, 
$i \le M$, there exist $M$ pairwise disjoint sets 
$\cH_i$ such that $\cH_i \subseteq \cJ_i$ and $|\cH_i| = K_i$, 
$i \le M$.
For the sets $\cJ_i$, let $D$ be the integer satisfying 
$y/x \le D < 1 + y/x$.
As $D < 1 + \log x$ and $y \le x\log x$, a suitably strong 
version of the prime number theorem for arithmetic progressions 
(cf.\ \cite[\S22 (4)]{DAV}) yields, for 
$(v_i,v_i + x/\log x] \subseteq (x,y]$, 
\[
 \sums[v < p \le v + x/\log x][p \equiv 1 \pod{D}] 1
  =
   \frac{x}{\phi(D)(\log x)^2} + O\bigg(\frac{y}{(\log y)^5}\bigg)
    \ge 
     \frac{x}{(\log x)^3} + O\bigg(\frac{x}{(\log x)^4}\bigg).
\]
As $K \le \log x$, we see that if $x$ is sufficiently large, 
then there are at least $K$ primes $p \equiv 1 \pod{D}$ in 
$(v_i,v_i + x/\log x]$.
If $q < q'$ are any two such primes, as $q' - q \le y$ 
and $q' \equiv q \pod{D}$, any prime divisor $p$ of $q' - q$ 
must either divide $D < 1 + \log x$ or be less than or equal to 
$y/D \le x$.
Hence $q' - q$ is $x$-smooth.
\end{proof}

\begin{proof}[Proof of Theorem \ref{thm:general}]
Fix $\theta \in (0,1]$ and a function 
$f_1 : [T_1,\infty) \to [x_1,\infty)$ 
of the first kind, and suppose Hypothesis \ref{hyp:EH} holds.
In accordance with Theorem \ref{thm:BFM4.3} (in which we take 
$a = 1$), let $K = K_{\theta}$ be a sufficiently large integer 
multiple of $\lceil (2/\theta)  \rceil + 1$, let $A = A_K$ be 
sufficiently large and  $\delta = \delta_{\theta}$ and 
$\eta = \eta(A,K)$ be sufficiently small.

In accordance with Corollary \ref{cor:thm2}, let $C$ be a 
sufficiently large but fixed positive constant and let $x$ be a 
sufficiently large number.
%
% $y \defeq cx\log x \log_3 x/\log_2 x$
%
Suppose, as we may, that $Cx \ge x_1$, and 
(cf.\ Definition \ref{def:w} (i)) set 
\[
   N \defeq (f_1^{-1}(Cx))^{1/\eta}.
\]
Thus, $Cx = f_1(N^{\eta})$ and $N$ tends to infinity with $x$.
Suppose $x$ is large enough so that in accordance with 
Theorem \ref{thm:BFM4.3}, $N \ge N(T_1,K,\eta)$.

By Definition \ref{def:w} (iv) there exists $L_{\eta} \ge 1$ such 
that $f_1(N)/x \to CL_{\eta}$ as $x \to \infty$.
Fix nonnegative real numbers 
$\alpha_1,\ldots,\alpha_{\lceil (2/\theta)  \rceil + 1}$ with 
$\alpha_1 \le \cdots \le \alpha_{\lceil (2/\theta)  \rceil + 1}$ 
and set 
\begin{equation}
 \label{eq:betaeta}
 \beta_i \defeq \alpha_iC L_{\eta}, 
  \quad i \le \lceil (2/\theta)  \rceil + 1.
\end{equation}

Fix a function $f_2 : [x_2,\infty) \to [z_2,\infty)$ of the second 
kind and consider the intervals
\begin{equation}
 \label{eq:intervals}
 (x + \beta_i f_2(x),x + \beta_i f_2(x) + x/\log x],
  \quad 
   i \le \lceil (2/\theta)  \rceil + 1.
\end{equation}
Recall that $y \defeq cx\log x \log_3 x/\log_2 x$, where $c > 0$ 
is a certain constant (cf.\ \eqref{eq:fgkmt3.1}). 
By Definition \ref{def:2ndkind} (ii) we have 
$x/\log x = o(f_2(x))$ and $f_2(x) = o(y)$.
Suppose, then, that $x$ is large enough (in terms of 
$\beta_{\lceil (2/\theta)  \rceil+1}$) so that the intervals in 
\eqref{eq:intervals} are all contained in $(x,y]$. 

In accordance with Lemma \ref{lem:goodktuple}, choose 
$\lceil (2/\theta)  \rceil + 1$ pairwise disjoint sets of primes 
$\cH_i$ of equal size, with 
\begin{equation}
 \label{eq:rightsize}
 \cH_i 
  \subseteq  
   (x + \beta_i f_2(x),x + \beta_i f_2(x) + x/\log x],
  \quad 
   i \le \lceil (2/\theta)  \rceil + 1. 
\end{equation}
Thus, letting 
\[
 \cH 
  \defeq 
   \cH_1 \cup \cdots \cup \cH_{\lceil (2/\theta)  \rceil + 1}
    \eqdef \{q_1,\ldots,q_K\},
\]
we have that $\prod_{1 \le i < j \le K}(q_j - q_i)$ is $x$-smooth, 
and hence $f_1(N^{\eta})$-smooth (we may suppose that $C \ge 1$).

As $K$ is fixed we may of course suppose that $K \le \log x$ and 
$p_K \le x$, so that \eqref{eq:Kbnd} is satisfied and $\cH$, being 
a set of $K$ primes larger than $p_K$, is admissible.
We may of course also suppose that $y \le N$, so that 
$\cH \subseteq [0,N]$.
Thus, $\cH$ satisfies each of the hypotheses of 
Theorem \ref{thm:BFM4.3}.

Let $\cZ(N^{\eta})$ be as in Hypothesis \ref{hyp:EH}, so that 
$\cZ(N^{\eta})$ is repulsive (cf.\ Definition \ref{def:sparse}), 
and note that since $x \le f_1(N^{\eta}) \le \log N^{\eta}$ (cf.\ 
Definition \ref{def:w} (ii)), 
\[
 \sums[p \in \cZ(N^{\eta})][p \ge p']
  \frac{1}{p}
   \ll
    \frac{1}{p'}
     \ll
      \frac{1}{\log_2 N^{\eta}}
       \le 
        \frac{1}{\log f_1(N^{\eta})}
         \le 
          \frac{1}{\log x}.
\]
Thus, \eqref{eq:Zsparse} is satisfied with 
$\cZ = \cZ(N^{\eta})$.
Therefore, by Corollary \ref{cor:thm2} there exists a vector of 
residue classes 
$
(c_p \pod{p})_{p \le f_1(N^{\eta}), \, p \, \not\in \, \cZ(N^{\eta})}
$ 
such that 
\begin{equation}
 \label{eq:Hsurvives}
 \cH
  =
   \big(\ZZ \cap (x,y]\big)
    \setminus \, 
     {\textstyle \bigcup_{p \le f_1(N^{\eta}), \, p \, \not\in \, \cZ(N^{\eta})}} 
       \, c_p \pod{p}.
\end{equation} 

Let $b \pod{W}$ be the arithmetic progression modulo 
\[
 W \defeq \prods[p \le f_1(N^{\eta})][p \not\in \cZ(N^{\eta})] p
\]
such that $b \equiv  -c_p \pod{p}$ for all primes 
$p \le f_1(N^{\eta})$ with $p \not\in \cZ(N^{\eta})$.
By \eqref{eq:Hsurvives} we have $(\prod_{i=1}^K(b + q_i),W) = 1$.

Each hypothesis of Theorem \ref{thm:BFM4.3} now accounted for, we 
conclude that there is some $n \in (N,2N] \cap b \pod{W}$, and a 
pair of indices 
$i_1,i_2 \in \{1,\ldots,\lceil (2/\theta)  \rceil + 1\}$, 
$i_1 < i_2$, such that 
\[
 \#(\bP \cap n + \cH_{i_1}) \ge 1
  \quad
   \text{and}
    \quad 
   \#(\bP \cap n + \cH_{i_2})
    \ge  
     1.
\] 
If there are more than two such indices, we take $i_2 - i_1$ to 
be minimal.

Thus, if $p$ is the largest prime in $\bP \cap n + \cH_{i_1}$ and 
$p'$ is the smallest prime in $\bP \cap n + \cH_{i_2}$, then $p$ 
and $p'$ are {\em consecutive}, that is, $p = p_t$ and 
$p' = p_{t+1}$ for some $t$.
Indeed, by \eqref{eq:Hsurvives} and the definition of 
$b \pod{W}$, for any $n \equiv b \pod{W}$ with 
$n + x \ge f_1(N^{\eta})$ we have
\[
 \bP \cap (n + x,n + y] = \bP \cap n + \cH. 
\]

By \eqref{eq:rightsize} and \eqref{eq:betaeta}, and since 
$x/\log x = o(f_2(x))$, we have 
\[
 p_{t+1} - p_t
  =
  (\beta_{i_2} - \beta_{i_1})f_2(x) + O\Big(\frac{x}{\log x}\Big)
   =
    (\alpha_{i_2} - \alpha_{i_1} + o(1))CL_{\eta}f_2(x).   
\]
Since there are only $O(1/\theta^2)$ distinct pairs of indices 
from which $i_1$ and $i_2$ may be chosen, we deduce that there 
exists a single pair $i_1 < i_2$ such that, for arbitrarily large  
$N$, we have    
\[
 p_{t+1} - p_t 
  = (\alpha_{i_2} - \alpha_{i_1} + o(1))CL_{\eta}f_2(x),
\]
for some pair of consecutive primes 
$p_t,p_{t+1} \in (N,N + y] \subseteq (N,3N]$.

Finally, using Definition \ref{def:w} (i), (iii) and (iv) and 
Definition \ref{def:2ndkind} (i) and (iii), we find that 
$
 CL_{\eta} f_2(x)\sim f_2(f_1(N)) \sim f_2(f_1(3N)).
$ 
We conclude that 
\[
 \frac{p_{t+1} - p_t}{f_2(f_1(p_t))}
  = 
   (1 + o(1))(\alpha_j - \alpha_i).
\]
We deduce \eqref{eq:genthm2} and \eqref{eq:genthm3} by using the 
argument of \cite[Corollary 1.2]{BFM}.
\end{proof}

As in \cite[Theorem 1.3]{BFM}, we may also consider ``chains'' of 
normalized, consecutive gaps between primes. 
Using essentially the same argument as above, but using (the 
unconditional) Theorem 4.3 (ii) of \cite{BFM} in place of 
Theorem \ref{thm:BFM4.3}, one may verify the 
following result. 

\begin{theorem}
 \label{thm:chains}
Fix any integer $a$ with $a \ge 2$.
Fix functions $f_1 : [T_1,\infty) \to [x_1,\infty)$ and 
$f_2 : [x_1,\infty) \to [z_2,\infty)$ of the first and second 
kinds respectively, and let $f \defeq f_2 \circ f_1$.
Let $d_n \defeq p_{n+1} - p_n$, where $p_n$ denotes the $n$th 
smallest prime, and let $\LP[a,f]$ denote the set of limit 
points in $[0,\infty]^a$ of the sequence of ``chains'' 
\[
 \textstyle 
%  \Big( 
  \Big( 
   \frac{d_n}{f(p_n)},\ldots,\frac{d_{n + a - 1}}{f(p_{n + a -1})}
  \Big)
%  \Big)_{p_n \ge T_1}.
\]
for $p_n \ge T_1$.
Given 
$\boldsymbol{\alpha} = (\alpha_1,\ldots,\alpha_K) \in \RR^K$, let 
$S_a(\boldsymbol{\alpha})$ be the set 
\[
  \big\{ 
  \big( 
   \alpha_{J(2)} - \alpha_{J(1)},
    \ldots,
     \alpha_{J(a+1)} - \alpha_{J(a)}
  \big)
    :
     1 \le J(1) < \cdots < J(a + 1) \le K
  \big)
 \big\}.
\]
For any $8a^2 + 16a$ nonnegative real numbers 
$\alpha_1 \le \cdots \le \alpha_{8a^2 + 16a}$, we have 
\[
 S_a(\boldsymbol{\alpha})  
  \cap 
   \LP[f]^a
    \ne 
     \emptyset.
\]
\end{theorem}

Let us call a function ``reasonable'' if it is of the form 
$f_2 \circ f_1$, where $f_1$ is a function of the first kind and 
$f_2$ is a function of the second kind. 
Theorem \ref{thm:chains} shows that for any $a$ there are 
infinitely many chains of consecutive prime gaps with 
$d_n,\ldots,d_{n + a - 1} > f(p_n)$ for any reasonable function 
$f$.
There are reasonable functions $f$ for which $f(T)/(R(T)\log_3 T)$ 
tends to $0$ arbitrarily slowly (recall that 
$R(T) = \log T\log_2 T\log_4 T/(\log_3 T)^2$ is the 
Erd{\H o}s--Rankin function).
We believe that in a forthcoming paper \cite{FMT}, Ford, Maynard 
and Tao show that for any $a$ there are infinitely many chains of 
consecutive prime gaps with 
$d_n,\ldots,d_{n + a - 1} \gg R(p_n)\log_3 p_n$.

%%%%%%%%%%%%%%%%%%%%%%%%%%%%%%%%%%%%%%%%%%%%%%%%%%%%%%%%%%%%%%%%%%
%%%%%%%%%%%%%%%%%%%%%%%%%%% REFERENCES %%%%%%%%%%%%%%%%%%%%%%%%%%%
%%%%%%%%%%%%%%%%%%%%%%%%%%%%%%%%%%%%%%%%%%%%%%%%%%%%%%%%%%%%%%%%%%

\end{document}